\documentclass [12pt] {article}
\usepackage{amsfonts}
\usepackage{amsmath}
\usepackage{amssymb}
\usepackage{hyperref}
\usepackage{bm,upgreek}
\usepackage[square,numbers]{natbib}
\usepackage{fullpage}
\usepackage{theorem}
\usepackage{authblk}
\usepackage{mathabx}

\newtheorem{thm}{Theorem}[section]

\newtheorem{lem}[thm]{Lemma}
\newtheorem{prop}[thm]{Proposition}
\newtheorem{defn}[thm]{Definition}
\newtheorem{rem}[thm]{Remark}
\theorembodyfont{\rmfamily}
\numberwithin{equation}{section}
\newenvironment{proof}{\noindent \emph{Proof.}}{\hspace{\stretch{1}}$\Box$}

\newcommand{\cL}{\mathcal{L}}
\newcommand{\cG}{\mathcal{G}}
\newcommand{\cV}{\mathcal{V}}
\newcommand{\cW}{\mathcal{W}}
\newcommand{\cH}{\mathcal{H}}
\newcommand{\cD}{\mathcal{D}}
\newcommand{\na}{\nabla}
\newcommand{\om}{\omega}

\newcommand{\bp}{\boldsymbol{p}}
\newcommand{\bc}{\boldsymbol{c}}
\newcommand{\cT}{\mathcal{T}}
\newcommand{\ce}{\mathcal{E}}
\newcommand{\cE}{\mathcal{E}}
\newcommand{\si}{\sigma}
\newcommand{\Rho} {\mathrm{P}}
\newcommand{\V} {\mathbb{V}}

\newcommand{\wh} [1] {\widehat{#1}}
\newcommand{\wt} [1] {\widetilde{#1}}
\newcommand{\mbf} [1]{\mathbf{#1}}
\newcommand{\mbb} [1]{\mathbb{#1}}
\newcommand{\mc} [1]{\mathcal{#1}}

\newcommand{\SL} {\mathrm{SL}}
\newcommand{\GL} {\mathrm{GL}}

\newcommand{\SO} {\mathrm{SO}}

\newcommand{\dd} {\mathrm{d}}
\newcommand{\ee} {\mathrm{e}}
\newcommand{\mr}[1] {\mathring{#1}}
\newcommand{\End} {\mathrm{End}}

\newcommand{\g} {\mathfrak{g}}

\newcommand{\T} {\mathbb{T}}
\newcommand{\R} {\mathbb{R}}

\newcommand{\X} {\mathbb{X}}

\usepackage{tensor}
\usepackage{easybmat}
\usepackage{color}
\usepackage{etex}
\usepackage[all]{xy}

\newcommand{\lpl}{
  \mbox{$
  \begin{picture}(12.7,8)(-.5,-1)
  \put(2,0.2){$+$}
  \put(6.2,2.8){\oval(8,8)[l]}
  \end{picture}$}}

\def\sideremark#1{\ifvmode\leavevmode\fi\vadjust{\vbox to0pt{\vss
 \hbox to 0pt{\hskip\hsize\hskip1em
 \vbox{\hsize3cm\tiny\raggedright\pretolerance10000
 \noindent #1\hfill}\hss}\vbox to8pt{\vfil}\vss}}}

\newcommand{\nn}[1]{(\ref{#1})}

\newcommand{\Addresses}{{
  \bigskip
  \footnotesize

  A.~R.~G. \, , \textsc{Department of Mathematics, The University of Auckland,
    Private Bag 92019, Auckland 1142, New Zealand}\par\nopagebreak
  \textit{E-mail address},   A.~R.~G. \, : \texttt{r.gover@auckland.ac.nz}

  \medskip

  D.~S. \, , \textsc{Department of Mathematics, The University of Auckland,
    Private Bag 92019, Auckland 1142, New Zealand}\par\nopagebreak
  \textit{E-mail address},   D.~S. \, : \texttt{daniel.snell@auckland.ac.nz}

  \medskip

  A.~T.-C. \, , \textsc{Department of Mathematics, Faculty of Arts and Sciences, American University of Beirut, P.O. Box 11-0236, Riad El Solh, Beirut 1107 2020, Lebanon}\par\nopagebreak
  \textit{E-mail address},   D.~S. \, : \texttt{at68@aub.edu.lb}

}}

\begin{document}

\title{Distinguished curves and integrability in Riemannian, conformal, and projective geometry}
\author{A.~Rod Gover, Daniel Snell, and  Arman Taghavi-Chabert}
\date{}
\maketitle


\begin{abstract}
  We give a new characterisation of the unparametrised geodesics, or
  distinguished curves, for affine, pseudo-Riemannian, conformal, and
  projective geometry. This is a type of moving incidence
  relation. The characterisation is used to provide a very general
  theory and construction of quantities that are necessarily conserved
  along the curves. The formalism immediately yields explicit formulae
  for these curve first integrals. The usual role of Killing tensors
  and conformal Killing tensors is recovered as a special case, but
  the construction shows that a significantly larger class of equation
  solutions also yield curve first integrals. In particular any normal
  solution to an equation from the class of first BGG equations can
   yield such a conserved quantity. For some equations the
  condition of normality is not required.

  For nowhere-null curves in pseudo-Riemannian and conformal geometry
  additional results are available. We provide a fundamental tractor-valued invariant of such curves and this quantity is parallel if and
  only if the curve is an unparametrised conformal circle.

\end{abstract}

\vspace{1pt}

\vspace{2pt}


\section{Introduction}

In the context of Riemannian geometry, geodesics were first defined as
the curves that minimise the distance between sufficiently close
points. Such curves $\gamma:I\to M$ are governed by
the equation
\begin{equation}\label{g1}
  \nabla_{\dot \gamma}\dot \gamma =0 ,
\end{equation}
where $\dot\gamma$ denotes the curve velocity, $\nabla$ is the
Levi-Civita connection, and $I\subset \mathbb{R}$ is an interval. This
  {\em geodesic equation} evidently makes sense and determines
distinguished curves on any manifold equipped with an affine
connection, that we also denote $\nabla$. (For simplicity all affine
connections will be assumed torsion free.)  In particular this applies
to pseudo-Riemannian geometries $(M,g)$ of any signature with $\nabla$
taken to be the Levi-Civita connection. In any such case the resulting
distinguished parametrised curves satisfying \nn{g1} are called {\em
    geodesics}. These play an essential role in the geometry and
analysis of manifolds and related physics, especially in connection
with general relativity \cite{Andersson-Blue,Frolov2017,Guill,HU,Wald1984}.

It is well known that symmetries can help understand and determine
geodesics. For example on a pseudo-Riemannian manifold $(M,g)$ a
vector field $k$ is called a {\em Killing vector field} if $\cL_k
  g=0$, where $\cL_k$ denotes the Lie derivative along the flow of
$k$. For such an infinitesimal automorphism $k$ it follows easily that
along any geodesic $\gamma$, the function $g(k,\dot\gamma)$ is
constant. Thus $g(k,\dot\gamma)$ is called a {\em first integral} of
$\gamma$. Higher rank Killing tensors and Killing-Yano tensor fields
(see Section \ref{BGG-sect} below), which are sometimes called {\em
    hidden symmetries}, can also lead to first integrals and these have,
for example, played an important role in the study of the
Kerr, Kerr-NUT-(A)dS and Pleba\'{n}ski-Demia\'{n}ski metrics, and related issues including black hole stability
\cite{Andersson-Blue,Carter1968,Frolov2017,IntegrabilityKillingEqn,KillingConstantsMotion}.
One key point is that if
enough first integrals are available then given a point and a
direction one can completely determine the trace of the curve with
that data. (For a given curve $\gamma:I\to M$ by its {\em trace} we mean its image $\gamma(I)$ in the manifold.)  This is the case for the metrics just mentioned. At
an extreme of this theme there is considerable interest in so-called
{\em superintegrable} geometries where there are more than dim$(M)$
first integrals for any geodesic \cite{BEH,GHKW,KKM}.

In this article we produce the
first steps of a new general and uniform approach to producing such first
integrals. While the term ``hidden symmetry'' already suggests a
notion of symmetry that is not classically obvious we will explain in
Section \ref{BGG-sect} \ that the Killing, Killing tensor, and
Killing-Yano, equations are just a small part of a vast family of
similar (in a suitable sense) overdetermined PDEs that are known as
  {\em first BGG equations}. These PDEs are defined in Theorem
\ref{normp} following \cite{CD,CSS}.  We show that certain solutions
of any of these equations can yield first integrals, and the formalism
immediately yields explicit formulae for the conserved quantities.

Each geodesic first integral yields a
constraint on any geodesic trace, but not on its parametrisation. This
strongly suggests that as a first step we should describe
distinguished curves in a parametrisation independent way.
Treating this effectively is linked to  {\em projective differential
    geometry}. \ This is the geometry not of an affine manifold
$(M,\nabla)$,  but the weaker structure $(M,\bp)$, where
$\bp:=[\nabla]$ denotes an equivalence class of torsion-free affine connections
that share the same unparametrised geodesics.

There is no preferred connection on the tangent bundle of a projective
manifold $(M,\bp)$. However there is a canonical connection
$\nabla^\cT$ on a related bundle $\cT$ of rank just one greater \cite{BEG}. The
bundle $\cT$ is called the projective {\em tractor bundle} and
$\nabla^\cT$ is the {\em tractor connection}. The dual connection on
$\cT^*$ is also called the tractor connection and these yield in an
obvious way a tractor connection on the respective tensor powers of
these and the tensor products thereof. These are the basic objects of
the invariant calculus for projective geometry that we introduce in
Section \ref{proj-sect}. Throughout the article, a \emph{$k$-tractor}
will refer to a section of the $k^{\rm th}$ exterior power  $\Lambda^k \mc{T}$ of the
tractor bundle, and we use $\wedge$ to indicate the exterior product of sections of such bundles.

It is useful here to
note that the link between $\cT$ and the tangent bundle $TM$ is via a
canonical sequence
\begin{equation}\label{euler*}
  0\to \ce(-1) \stackrel{X}{\to}\cT\to TM(-1) \to 0
\end{equation}
where the density bundle $\ce(-1)$ is a suitable root of the square
of the top exterior power of $TM$, and $TM(-1)$ means $TM\otimes
  \ce(-1)$. The bundle map $X$, which can alternatively be thought of as a
section of $\cT(1)=\cT\otimes \ce(-1)^*$, is called the {\em canonical tractor}.
This plays an important role and, over a point $x\in M$, it invariantly encodes  information
concerning the position of that point relative
to other geometric data.

Throughout we will only consider curves with trace a connected smoothly embedded
1-manifold.
We can now
state one of the first main results.
\begin{thm}\label{main-p}
  On an affine or projective manifold an unparametrised oriented  curve
  $\gamma$ is an unparametrised oriented geodesic if and only if along  $\gamma$ there is a parallel projective $2$-tractor $0\neq \Sigma \in \Gamma(\Lambda^2 \cT|_\gamma)$ such that
  \begin{equation}\label{main-p-eqn}
    X\wedge \Sigma=0.
  \end{equation}
  For a given unparametrised oriented geodesic  $\gamma$ the $2$-tractor  $\Sigma_\gamma$ satisfying \nn{main-p-eqn} is unique up to multiplication by a positive constant.
\end{thm}

There is also considerable interest in the conformal analogues of the
Killing equation and its generalisations
\cite{Bochner1948,Yano1952,Kashiwada1968,Tachibana1969,Semmelmann2003,G-Sil-ckforms,Mason2010,Dunajski2010}. Again first integrals provide one
motivation. On a pseudo-Riemannian manifold $(M,g)$, a vector field
$k$ is said to be a {\em conformal Killing vector} field if
$\mathcal{L}_kg=\rho g$ (for some function $\rho$). For such a vector
field $g(k,\dot \gamma)$ is a first integral for any parametrised {\em
    null geodesic} $\gamma$. Recall a curve $\gamma:I\to M$ is null if
its velocity $\dot\gamma \neq 0$ satisfies $g(\dot\gamma,\dot
  \gamma)=0$ everywhere along the curve. More generally similar first
integrals for geodesics that are null, in this way, arise from
conformal Killing tensors (see \eqref{cKt}), for example.

Treating the natural extension of these observations involves
conformal geometry.  A signature $(p,q)$ {\em conformal manifold}
consists of a pair $(M,\bc)$ where $\bc$ is an equivalence class of
signature $(p,q)$ metrics, where any two metrics $g,\hat{g} \in
  \bc$ are related by {\em conformal rescaling}, that is we have $\hat{g}=f g$ for some positive smooth function $f$. In
analogy with projective geometry, on conformal manifolds the basic
conformally invariant calculus is also based around an invariant
tractor bundle and connection, see Section \ref{csec} for details. To emphasise
similarities with the projective case, and also to simplify notation,
we denote these by essentially the same notation as in the projective
case. Because of context, no confusion should arise (and we do use a
different index set).  Thus $\cT$ will denote the standard conformal tractor
bundle and $\nabla^\cT$ the usual tractor connection on this, and $X$
denotes the (conformal) canonical tractor.

Upon conformal rescaling, null geodesics are simply
reparametrised. Thus, as unparam\-etrised curves, null geodesics are
among the distinguished curves of conformal manifolds $(M,\bc)$ of signature $(p,q)$ with
$pq\neq 0$.  These are characterised by a close analogue of Theorem \ref{main-p} as follows.
\begin{thm}\label{main-nullc}
  On a pseudo-Riemannian or conformal manifold, a curve $\gamma$ is an unparam\-etrised oriented
  null geodesic if and only if along $\gamma$ there is a
  parallel  conformal $2$-tractor $0\neq \Sigma \in \Gamma(\Lambda^2 \cT|_\gamma)$
  such that
  \begin{equation}\label{main-nullc-eqn}
    X\wedge \Sigma=0.
  \end{equation}
  For a given oriented null geodesic trace  $\gamma$ the $2$-tractor  $\Sigma_\gamma$ satisfying \nn{main-nullc-eqn} is totally null and unique up to multiplication by a positive constant.
\end{thm}
\noindent The notion of {\em totally null} used in the Theorem is characterised by the nilpotency condition given in expression \eqref{nil-eq}.

Null geodesics are a very restricted class of distinguished
curves. In particular, they are unavailable in the case of definite
signature. On a conformal manifold the nowhere-null distinguished curves
are the so-called conformal circles \cite{Yano1938,Schouten1954,FriedSchmidt,Bailey1990a}.
The differential equation for these
is somewhat more complicated than the geodesic equation \nn{g1}.
Fixing a metric $g\in \bc$, a
parametrised curve $\gamma:I\to M$ is said to be a {\em
    conformal circle} if it satisfies the (conformally invariant)
equation
\begin{equation}
  \label{par-conf-circ}
  u^b \nabla_b a^c  - 3 \frac{u \cdot a}{u \cdot u} a^c + \frac{3 \, a \cdot a}{2 \, u \cdot u} u^c - ( u \cdot u ) u^b \Rho_b {}^c + 2  \, \Rho_{ab} u^a u^b u^c=0 \, ,
\end{equation}
where $u= \dot \gamma$ and $a=\ddot \gamma$ and $g(u,u)\neq0$.
Parametrised conformal circles may be understood in terms of tractors
\cite{BEG}, and this provides some conceptual simplification and an
equation that, although third order, is similar in spirit to
\nn{g1}. We review this in Section \ref{Conf-circ-sect}. See also \cite{unique-MikE,Tod,SilhanVojtech} for alternative useful characterisations of these curves.

It is natural to investigate the possibility of first integrals for
conformal circles. Any na\"{\i}ve approach needs to confront two new
problems. First that the governing equation is of third order, so first
integrals should be expected to involve higher order objects. Second
there is the related issue of parametrisation.  Whereas geodesics have
a distinguished class of affine parametrisations the class of
distinguished parametrisations determined by \nn{par-conf-circ}, the
so-called projective parametrisations (see Section \ref{param}), is
larger.  Thus comparing to geodesics there is potentially an even
greater gain from a parametrisation free description. We shall say
that an unparametrised curve $\gamma$ is an unparametrised conformal
circle if it admits a projective parametrisation so that the resulting
curve $\gamma:I\to M$ satisfies \nn{par-conf-circ}.  Then in terms of
the conformal tractor bundle we have the following result.
\begin{thm}\label{main-c}
  On a pseudo-Riemannian or conformal manifold a nowhere null curve
  $\gamma$ is an oriented conformal circle if and  only if along  $\gamma$ there is a parallel $3$-tractor $0\neq \Sigma \in \Gamma(\Lambda^3 \cT|_\gamma)$ such that
  \begin{align}\label{main-c-eqn}
    X\wedge \Sigma=0.
  \end{align}
  For a given oriented conformal circle $\gamma$ the $3$-tractor $\Sigma_\gamma$
  satisfying \eqref{main-c-eqn} is unique up to multiplication by a
  positive constant, and unique if we specify $|\Sigma_\gamma|^2=-1$ when $\gamma$ is spacelike, or $|\Sigma_\gamma|^2=1$ when $\gamma$ is timelike.
\end{thm}

Specialising to the case of the homogeneous model for projective
geometry and then also the homogeneous model for conformal geometry
the condition $X\wedge \Sigma=0$ agrees with an incidence relation: In
each of these settings the tractor field $\Sigma$ may be taken to be
parallel not just along the distinguished curve it determines but
rather parallel everywhere. Then also, in each of these homogeneous models, the
canonical tractor $X$ can be identified with suitable homogeneous
coordinates for the underlying point. See Sections \ref{model},
\ref{nmodel}, and \ref{circ-model} for details.

Using the above theorems there is a simple route to certain first integrals. For
example, in the setting of Theorem \ref{main-p} or \ref{main-nullc},
suppose one has a section $\psi$ of $\otimes^s(\Lambda^2\cT^*)$ that is
parallel for the relevant tractor connection.  Then $\psi$ pairs
with $\otimes^s\Sigma$ to yield a geodesic first integral. There is a
similar observation for conformal circles that uses Theorem
\ref{main-c}. Parallel tractor fields correspond to certain solutions,
called normal solutions, of invariant overdetermined PDEs called first
BGG equations. See Theorem \ref{normp} in Section \ref{BGGsec}. Then
the universal construction of corresponding first integrals is treated
by Theorem \ref{fi-thm}; this is one of the main theorems here and applies simultaneously to the three
settings of the Theorems \ref{main-p}, \ref{main-nullc}, and \ref{main-c} above.
Despite its technical nature, this machinery produces explicit formulae for these first integrals, which may be verified to be invariant and conserved along the given distinguished curve.

Theorem \ref{fi-thm} is then illustrated by various examples in
sections \ref{class}, \ref{case-BGG2}, \ref{nullfi}, \ref{ob-ex}, and
\ref{S-ex}. Several of these examples may be considered special cases
of a general procedure for producing first intgerals that, for the
case of conformal circles, is outlined in Section \ref{gen_proc}. See
in particular Theorem \ref{genid} which shows that it is easy to use
Theorem \ref{fi-thm} proliferate non-trivial conformal circle first
integrals. Then the use of this Theorem is further illustrated in
Section \ref{genex}.   For the case of geodesics and null geodesics one
expects, by classical theory, the first integrals in all cases as
constructed to be linked to Killing tensors. This arises naturally in
the constructions here and the explicit link is described in
Proposition \ref{BGGtoK-prop} and Proposition
\ref{BGGtoK-cprop}. These explain how normal BGG solutions yield
normal Killing tensors and, respectively, normal conformal Killing
tensors.

Surprisingly the examples treated also lead to results that are,
in each case, stronger than that given by the general Theorem
\ref{fi-thm}, as follows.  Each of the examples treated exhibits first
integrals for geodesics, null geodesics, or conformal circles as
arising from various first BGG equations. The general theory of
Theorem \ref{fi-thm} requires that the solution be normal, in that it
corresponds to parallel tractor according to Theorem \ref{normp}. But
actually, for each of the BGG equations and first integrals treated
explicitly, the normality turns out to be {\em not required}. See
Remark \ref{non-normal}, Theorem \ref{3rd-thm}, Theorem
\ref{basic-ex-thm}, and Theorem \ref{S-thm}. This suggests the
interesting possibility that there may be a strengthening of Theorem \ref{fi-thm} in some generality.

Finally in this context we should mention that because the treatment
of the curves is parametrisation independent, and so also are the
first integrals constructed, the results apply to infinity on
appropriately compactifiable complete non-compact manifolds. For
example the projective treatment provides first integrals that extend
to the boundary at infinity of manifolds that are projectively compact
in the sense of \cite{CG-proj-Ein,CG-proj-confbound,CGH-jlms}. The conformal treatment
yields curves and first integrals that extend to the infinity of
conformally compact manifolds. This should be useful for extending the theories of superintegrability and separation of variables to such settings.

Some history is relevant here. Examples of conformal circle first
integrals were constructed and applied for specific classes of metrics
(and in some generality in dimension 4) by Tod in \cite{Tod}. Indeed
in this context the example of Section \ref{ob-ex} arises. We thank
Maciej Dunajski for pointing this out  and note that more
  recently Dunajski and Tod have applied the same first integral to
  find classes of conformal circles and even establish the complete
  integrability of the conformal circle equation on certain classes of
  4-manifold \cite{DT-pre}. First integrals for
parametrised conformal circles were looked at in the thesis works of
the second author \cite{Snell-honours,SnellMSc} using the tractor
approach from \cite{BEG}. A slightly different and parametrisation
free approach was developed by Bell \cite{BellMSc} and his work has
certainly influenced our development. He also gives another
characterisation of conformal circles in terms of a symmetric
2-tractor (that arises from our machinery in Theorem \ref{S-thm}
below). Recently \v Silhan and \v{Z}\'{a}dn\'{i}k \cite{SilhanVojtech}
have developed an interesting tractor Frenet theory for curves, and
associated with this recovered some first integrals in the same spirit
as those looked at by Bell and Snell.

There are additional results in the work here. In the case of nowhere-null curves
in conformal geometry we can canonically associate the $3$-tractor $\Sigma$, even if the curve is not distinguished. See Lemma
\ref{tri-lem}. Thus $\Sigma$ is a fundamental invariant of such curves
and so may be used to construct, in obvious ways, other invariants of
such curves. Indeed $\Sigma$ provides the full information of the curve.
Then finally in Section \ref{zero-sec} we show that
for normal first BGG equation solutions the zero locus of a suitable part of
the solution jet describes a distinguished curve. See Proposition \ref{cc-zero},
Proposition \ref{nullzero}, and Proposition \ref{f-zero}.

\smallskip

Section \ref{back}, and then Section \ref{proj-sect} up to Section
\ref{ptsec} present background material on affine and projective
geometry, including the tractor calculus. Similar background for
conformal geometry is found in Section \ref{c-sect} and Section
\ref{csec}. The ordinary differential equations describing the
parametrisation independent treatment of geodesics, null geodesics,
and then conformal circles are found in Section \ref{projeqsec},
Section \ref{null-sec}, and Section \ref{ccirc} respectively.
Theorem \ref{main-p}, Theorem \ref{main-nullc} and Theorem
\ref{main-c} are proved in Section \ref{projeqsec}, Section
\ref{null-sec}, and Section \ref{trac-form}, respectively.

\smallskip

Conformal and projective geometries are special cases of the large
class of parabolic geometries \cite{CS-book} and for these structures
there is a general theory of distinguished curves
\cite{DistCurvParabolic}. It seems likely that for all such distinguished
curves there will be a tractor based incidence type characterisation
of these curves that generalises the developments on this
article. This should also lead to generalisation of the results found here on first integrals.
This direction and other extensions of the work here will be treated elsewhere.

\smallskip

Throughout manifolds and tensors on manifolds will be assumed smooth.
When it is convenient, we will use standard abstract index notation (in
the sense of Penrose). For example we may write $\ce^a$ (respectively
$\ce_a$) for the tangent bundle $TM$ (respectively cotangent bundle
$T^*M$) of a manifold $M$ and $\xi^a$ (respectively $\omega_a$) for a
vector field (respectively a $1$-form field) on $M$. Then we write
$\xi^a\omega_a$ for the canonical pairing between vector fields and
$1$-forms and denote by the Kronecker delta $\delta^b{}_a$ the
identity section of the bundle $\textrm{End}(TM)$ of endomorphisms of
$TM$. Indices enclosed by round (respectively by square brackets)
indicate symmetrisation (respectively skew-symmetrisation) over the
enclosed indices. When tractor bundles are introduced these will also
be adorned with abstract indices when convenient.

\paragraph{Acknowledgments}
A. R. G. gratefully acknowledges support from the Royal
Society of New Zealand via Marsden Grant 16-UOA-051. A. T.-C. declares that this work was partially supported by the grant 346300 for IMPAN from the Simons Foundation and the matching 2015-2019 Polish MNiSW fund.  A. T.-C. was also supported by a long-term faculty development grant from
the American University of Beirut for his visit to IMPAN, Warsaw, in the summer 2018.

\section{Background and notation for affine  geometry}\label{back}

Let $(M,\nabla)$ be an  affine manifold (of dimension $n\geq
  2$), meaning that $\nabla$ is a torsion-free affine connection.
The curvature
$$R_{ab}{}^c{}_d\in\Gamma(\Lambda^2 T^*M \otimes TM\otimes T^*M )$$
of the
connection $\nabla$ is given by
$$
  [\nabla_a,\nabla_b]v^c=R_{ab}{}^c{}_d  v^d , \qquad v\in \Gamma(TM).
$$
The Ricci curvature is defined by $R_{bd} =R_{cb}{}^c{}_d $.

\subsection{Decomposition of curvature: projective} \label{projdec}

On an affine manifold the
trace-free part $W_{ab}{}^c{}_d$ of the curvature $R_{ab}{}^c{}_d$ is
called the {\em projective Weyl curvature} and we have
\begin{equation}\label{decp}
  R_{ab}{}^c{}_d= W_{ab}{}^c{}_d +2 \delta^c{}_{[a}\Rho_{b]d}+\beta_{ab}\delta^c_d,
\end{equation}
where $\beta_{ab}$ is skew and $\Rho_{ab}$ is called the {\em projective Schouten tensor}. That $W_{ab}{}^c{}_d$ is trace-free means exactly that $W_{ab}{}^a{}_d=0$ and $W_{ab}{}^d{}_d=0$.
Since $\nabla$ is torsion-free the Bianchi symmetry
$R_{[ab}{}^c{}_{d]}=0$ holds, whence
$$
  \beta_{ab}=-2\Rho_{[ab]} \qquad \mbox{and} \qquad (n-1)\Rho_{ab} = R_{ab}+\beta_{ab}.
$$

From the differential Bianchi identity we obtain that $\beta$ is closed and
\begin{equation}
  \nabla_c W_{ab}{}^c{}_{d}=(n-2)C_{abd},
\end{equation}
where
\begin{equation}\label{Cotton}
  C_{abc}:= \nabla_a \Rho_{bc}-\nabla_b \Rho_{ac}
\end{equation}
is called the {\em
    projective Cotton tensor}. In dimension 2 the projective Weyl tensor is identically zero.

As we shall see below the curvature decomposition \nn{decp} is useful
in projective differential geometry.

\section{Geodesics and projective geometry}\label{proj-sect}

As mentioned above for Riemannian, pseudo-Riemannian, or more
generally affine geometry, the treatment of unparametrised geodesics
involves projective differential geometry.

Some further notation is in order first. On a smooth $n$-manifold $M$ the bundle
$\mathcal{K}:=(\Lambda^{n} TM)^2$ is an oriented line bundle and thus
we can take correspondingly oriented roots of this. For projective
geometry a convenient notation for these is as follows: given $w\in
  \mathbb{R}$ we write
\begin{equation} \label{pdensities}
  \ce(w):=\mathcal{K}^{\frac{w}{2n+2}} .
\end{equation}
\newcommand{\cK}{\mathcal{K}}

\subsection{Projective geometry}

Two affine connections $\nabla$ and $\widehat{\nabla}$ on a manifold are
said to be {\em projectively equivalent} if they have the same
geodesics as unparameterised curves. Any two connections that differ
only by torsion are projectively equivalent, and thus in the study of
projective differential geometry it is usual to work with torsion-free
connections. Two such torsion-free connections $\nabla$ and $\widehat{\nabla}$
are projectively equivalent if and only if there exists a $1$-form $\Upsilon\in\Gamma(T^*M)$ such that
\begin{equation}\label{projective_change}
  \widehat{\nabla}_a\xi^b=\nabla_a\xi^b+\Upsilon_a\xi^b+\delta^b{}_a\Upsilon_c\xi^c.
\end{equation}

\begin{defn}
  A manifold $M$ of dimension
  $n\geq 2$ equipped with is an equivalence class $\bp$ of projectively equivalent torsion-free affine connections is called a \underline{projective manifold}.
\end{defn}

The standard homogeneous model for oriented projective manifolds is
the $(n+1)$-dimen\-sional sphere arising as the ray projectivisation
$S^{n}:=\mathbb{P}_+(\mathbb{R}^{n+1})$ of $\mathbb{R}^{n+1}$
(i.e.\ the double cover of $\mathbb{RP}^{n}$). The
group of orientation-preserving projective diffeomorphisms of
$S^{n}$ can be identified with the special linear group
$\SL(n+1,\mathbb{R})$ acting transitively on
$\mathbb{P}_+(\mathbb{R}^{n+1})$ in the standard way.

Any affine connection $\nabla$ induces a connection on the bundle
$\cK$ and hence a connection on the bundles $\ce(w)$ of projective
densities. For two projectively equivalent affine connections
$\nabla$ and $\widehat{\nabla}$ related as in \eqref{projective_change} their
induced connections on the bundles $\cE( w )$
are related by
\begin{equation}\label{projective_change_on_densities}
  \widehat{\nabla}_a\sigma=\nabla_a\sigma+ w \Upsilon_a \sigma \, .
\end{equation}

\subsection{The projective tractor bundle and connection}\label{ptsec}

As mentioned in the introduction, on a general projective $n$-manifold
$(M,\bp)$ there is no distinguished connection on $TM$. However there
is a projectively invariant connection on a related rank $(n+1)$
bundle $\cT$. This is the projective tractor connection (of \cite{ThomasProj}) that we now
describe following \cite{BEG,CGMacbeth-Ein}.

Consider the first jet prolongation
$J^1\ce(1)\to M$ of the density bundle $\ce(1)$. (See for example
\cite{palais} for a general development of jet bundles.)
There is a canonical bundle map called the {\em jet projection map}
$J^1\ce(1)\to\ce(1)$, which at each point is determined by the map
from 1-jets of densities to simply their evaluation at that point, and
this map has kernel $T^*M (1)$.  We write $\cT^*$, or in an
abstract index notation $\ce_\alpha$, for $J^1\ce(1)$ and $\cT$ or $\ce^\alpha$
for the dual vector bundle. Then we can view the jet projection as a
canonical section $X^\alpha$ of the bundle $\ce^\alpha(1)$. Likewise,
the inclusion of the kernel of this projection can be viewed as a
canonical bundle map $\ce_a(1)\to\ce_\alpha$, which we denote by
$Z_\alpha{}^a$. Thus the jet exact sequence (at 1-jets) is written in this
notation as
\begin{equation}\label{euler}
  0\to \ce_a(1)\stackrel{Z_\alpha{}^a}{\to} \ce_\alpha \stackrel{X^\alpha}{\to}\ce(1)\to 0.
\end{equation}
We write $\ce_\alpha=\ce(1)\lpl \ce_a(1)$ to summarise the composition
structure in \nn{euler} and $X^\alpha\in \Gamma(\ce^{\alpha}(1))$, as defined in
\nn{euler}, is called the {\em canonical tractor}. The sequence \eqref{euler*} is the dual to \eqref{euler}.

As mentioned above,
any connection $\nabla \in \bp$
determines a connection on $\ce(1)$ (and vice versa). On the other hand,  by definition,
a connection on $\ce(1)$  is precisely a splitting
of the 1-jet sequence \nn{euler}.
Thus given such a choice we have the direct sum
decomposition $\ce_\alpha \stackrel{\nabla}{=} \ce(1)\oplus \ce_a(1) $
and we write
\begin{equation}\label{split}
  Y_\alpha:\ce(1) \to \ce_\alpha \qquad \mbox{and} \qquad W^\alpha{}_a: \ce_\alpha\to \ce_a(1),
\end{equation}
for the bundle maps giving this splitting of \nn{euler}; so
$$
  X^\alpha Y_\alpha=1, \qquad  Z_\alpha{}^b W^\alpha{}_a=\delta^b {}_a, \qquad \mbox{and} \qquad Y_\alpha W^\alpha{}_a=0.
$$

\newcommand{\bX}{\mathbb{X}}

Observe that any bundle map $ \ce_\alpha\to \ce_a(1)$ that splits the
sequence \eqref{euler} must differ from $W^\alpha{}_a$ by a section
that takes the form $X^\alpha\otimes \Upsilon_b$ for some 1-form
$\Upsilon_b$. Thus from \eqref{euler} we deduce that there is a projectively invariant injective
bundle   map from $TM(-2)$ into $\Lambda^2\cT$:
\begin{equation}\label{injp}
  \bX^{\alpha \beta}_b: \ce^b(-2)\to \ce^{[\alpha\beta]} \quad\mbox{given by} \quad \mbf{v}^b\mapsto 2X^{[\alpha} W^{\beta ]} {}_b \mbf{v}^b.
\end{equation}

With
respect to a splitting \nn{split}
we define a connection on $\cT^*$ by
\begin{equation}\label{pconn}
  \nabla^{\mathcal{T}^*}_a \binom{\si}{\mu_b}
  := \binom{ \nabla_a \si -\mu_a}{\nabla_a \mu_b + \Rho_{ab} \si}.
\end{equation}
Here $\Rho_{ab}$ is the projective Schouten tensor of $\nabla\in \bp$, as introduced earlier.
It turns out that \nn{pconn} is
independent of the choice $\nabla \in \bp$, and so the  {\em cotractor connection}
$\nabla^{\mathcal{T}^*}$ is determined canonically by the projective
structure $\bp$.
Thus we shall
also term $\cT^*=\ce_\alpha$ the {\em cotractor bundle}, and we note the dual
  {\em tractor bundle} $\cT=\ce^\alpha$
has canonically the dual {\em tractor connection}: in terms of a
splitting dual to that above this is given by
\begin{equation}\label{tconn}
  \nabla^\cT_a \left( \begin{array}{c} \nu^b \\
      \rho
    \end{array}\right) =
  \left( \begin{array}{c} \nabla_a\nu^b + \rho \delta^b {}_a \\
      \nabla_a \rho - \Rho_{ab}\nu^b
    \end{array}\right).
\end{equation}
Note that given a choice of $\nabla\in \bp$, by coupling with the
tractor connection we can differentiate tensors taking values in
tractor bundles and also weighted tractors. In particular
we have
\begin{equation}\label{trids}
  \nabla_aX^\beta=W^\beta{}_a,  \quad \nabla_a W^\beta{}_b=-\Rho_{ab} X^\beta, \quad  \nabla_a Y_\beta =\Rho_{ab}Z_\beta{}^b,  \quad\mbox{and}\quad  \nabla_a Z_\beta {}^b = -\delta^b {}_a Y_\beta.
\end{equation}
These encode the tractor connection and the formula for connection acting  on a section of any
tractor bundle can then be deduced from the Leibniz rule.

Finally, the \emph{tractor curvature} $\Omega_{ab}{}^{\gamma}{}_{\delta}
$ of the tractor connection, defined by $\Omega_{ab}{}^{\gamma}{}_{\delta} \Phi^{\delta} := 2 \, \nabla_{[a} \nabla_{b]}
  \Phi^{\gamma}$, for any $\Phi^\gamma \in \Gamma( \mc{T} )$, can be expressed in a splitting as
\begin{align}\label{ptract-curv}
  \Omega_{ab}{}^{\gamma}{}_{\delta} & = W_{ab}{}^c{}_d W^\gamma {}_c Z_\delta {}^d - C_{abc} Z_\delta {}^c X^\gamma
\end{align}
A projective structure is said to be \emph{(locally) flat} if this
tractor curvature vanishes, in which case the manifold is locally diffeomorphic to $S^n$.

\subsection{The parametrisation independent treatment of geodesics}\label{parsec}

Throughout this article, a {\em curve} on a smooth manifold $M$ will
mean a connected one-dimensional embedded submanifold $\gamma$ of $M$, where
this submanifold is identified with its image under the
embedding. (While curves may be defined in a way that defines more
general objects, locally our definition imposes no restriction.)  A
  {\em velocity (vector) field} $u^a$ along $\gamma$ is a nowhere
vanishing section of the tangent bundle $T \gamma$ of $\gamma$.  A
corresponding parametrisation of the curve $\gamma$ is a choice of
real-valued smooth function on $\gamma$ that satisfies $u^a \nabla_a t
  = 1$.
Given the velocity field $u^a$ the parameter $t$ may exist only locally. But locally  it is determined
up to the addition of  a constant and is equivalent to
a map $I\to \gamma$, where $I$ is a real interval, with $u^a$ the push-forward of $\frac{\dd}{\dd t}$. By a slight abuse of terminology we shall call the data $(\gamma,u^a)$  a
  {\em parametrised curve}.

An \emph{oriented} curve is a curve together with a choice of orientation, that is, a nowhere-vanishing $1$-form on $\gamma$. This is clearly equivalent to a choice of velocity field along $\gamma$. In particular, a choice of parametrisation endows a curve with an orientation. If the orientation of a curve is defined by its velocity $u^a$, then the velocity field $-u^a$ defines the opposite orientation.

Note that vector and tractor bundles on $M$ can be pulled back, i.e.\ restricted, to $\gamma$. This will be reflected in the notation in this article where restriction of a bundle $\mathcal{B}$ to $\gamma$ will be denoted $\mathcal{B}|_\gamma$. In particular, the tangent bundle $T \gamma$ is a subbundle of the restriction $T M|_\gamma$ of the tangent bundle $T M$ to $\gamma$. We shall also write $T \gamma (w)$ for $T\gamma \otimes \mc{E}(w)|_\gamma$ for any weight $w \in \R$.

\subsection{Distinguished curves in projective geometry}\label{projeqsec}

Let $(M,\nabla )$ be an affine manifold.
A parametrised geodesic for $(M,\nabla )$ is a parametrised curve with
velocity $u^a$ satisfying
\begin{align}\label{eq-aff-par-geod}
  u^b \nabla_b u^a & =0 \, .
\end{align}
As a temporary simplification, let us suppose that $\nabla$ is \emph{special},  i.e.\  $\nabla$ preserves a chosen volume density.
If we change to a projectively related special affine connection
$\widehat{\nabla}$  then we have (\ref{projective_change}) with
$\Upsilon_a = \si^{-1}\nabla_a \si$ for some $\si \in \ce_+(1)$, where $\ce_+(1)$ denotes the ray subbundle of $\ce(1)$ consisting of strictly positive densities. We will call such a $\si$ a \emph{scale}.
Thus
$$
  u^b \widehat{\nabla}_b u^a  = 2u^a\Upsilon_bu^b,
$$ and so in this sense the geodesic equation is not projectively
invariant. However there is a commensurate reparametrisation of
$\gamma$, equivalently a rescaling of the velocity $u^a$: if we
reparametrise $\gamma$ so that with this new parametrisation it has
velocity $\widehat{u}^a = \si^{-2}u^a$  then
$$
  \widehat{u}^b \widehat{\nabla}_b \widehat{u}^a =0.
$$ This suggests a projectively invariant geodesic equation that we
now describe in the general setting.

Before we proceed, we introduce some terminology. Given a projective
manifold $(M,\bp)$, we say that a curve $\gamma$ is an {\em
    unparametrised geodesic} in $(M,\bp)$ if, given $\nabla\in \bp$,
$\gamma$ admits a velocity field $u$ satisfying
(\ref{eq-aff-par-geod}). For a given unparametrised geodesic it is clear that
the velocity field depends on the choice of $\nabla\in \bp$ and it is
easily verified that the property of being an unparametrised geodesic is
independent of parametrisation.

\begin{lem}\label{lem-wt-proj-geod}
  Let $\gamma$ be an oriented curve on $(M,\bp)$. Then $\gamma$ is an unparametrised
  geodesic, with respect to $\bp$, if and only if there exists
  a non-vanishing vector field $\mbf{u}^a\in \Gamma(T\gamma (-2))$ along
  $\gamma$ satisfying the projectively invariant equation
  \begin{align}\label{eq-proj-inv-geod}
    \mbf{u}^b \nabla_b \mbf{u}^a & = 0 \, , & \mbox{where} & \,  \nabla\in \bp.
  \end{align}
  The weighted velocity  field $\mbf{u}$ is unique up to a positive factor that is constant along $\gamma$.
\end{lem}

\begin{proof}
  It follows at once from (\ref{projective_change}) and (\ref{projective_change_on_densities}) that (\ref{eq-proj-inv-geod}) is projectively invariant.

  Suppose $\gamma$ is an unparametrised geodesic for $(M,\bp)$, and
  $u^a$ is a smooth non-vanishing vector field tangent to $\gamma$ that
  is consistent with the orientation.  Let $\nabla\in \bp$. From
  \eqref{projective_change} it follows that $u^b \nabla_b u^a = f u^a$
  for some smooth  function $f$ along $\gamma$. Working locally let
  $\sigma \in \Gamma( \mc{E}_+(1)|_\gamma )$ solve $2 \, \sigma^{-1} u^a
    \nabla_a \sigma = f$.  Then $\mbf{u}^a := \sigma^{-2} u^a \in
    \Gamma( \mc{E}^a(-2)|_\gamma )$ satisfies \eqref{eq-proj-inv-geod}. One can
  easily check that the resulting $\mbf{u}^a$ is independent of the
  initial choice of oriented parametrisation.  Conversely, given any
  solution $\mbf{u}^a$ of \eqref{eq-proj-inv-geod} and any density
  $\sigma \in \Gamma( \mc{E}_+(1)|_\gamma )$ along $\gamma$ one obtains a
  vector field $u^a := \sigma^2 \mbf{u}^a$ tangent to $\gamma$ that
  satisfies an equation $u^b \nabla_b u^a = f u^a$ where $f$ is
  smooth. Thus, locally (and hence globally), there is a reparametrisation with
  velocity $\widehat{u}^a$ satisfying $\widehat{u}^b \nabla_b
    \widehat{u}^a = 0$. The final statement is obvious.
\end{proof}

Note that given an unparametrised geodesic $\gamma$, with corresponding weighted
velocity $\mbf{u}$, a projective scale $\si\in \Gamma( \ce_+(1) )$ determines a
unique section $u^a = \sigma^2 \mbf{u}^a$ of $T \gamma$, and thus a
parametrisation of $\gamma$ up to an additive constant. Also
$\sigma \in \Gamma( \mc{E}_+(1) )$ determines an affine connection $\nabla_a$
in $\bp$, then with respect to this $u^a$  is tangent to $\gamma$ as an affinely
parametrised geodesic, i.e. it solves \eqref{eq-aff-par-geod}.

Next we observe that the equation \eqref{eq-proj-inv-geod} has a nice
interpretation in the tractor picture. Let $\gamma$ be a curve with
weighted velocity $\mbf{u}^b\in \Gamma(T\gamma(-2))$. Observe that $\mbf{u}$
determines a weight zero $2$-tractor, which we will denote $\Sigma \in \Gamma( \Lambda^2\cT|_\gamma)$, along $\gamma$ via the map
\eqref{injp}:
\begin{equation}\label{eq-geod-bitract-proj}
  \Sigma^{\alpha\beta}:= \bX^{\alpha\beta}_b\mbf{u}^b.
\end{equation}
Then we have:
\begin{prop}\label{prop-proj-geod-par}
  An oriented curve $\gamma$ on $(M,\bp)$ is an unparametrised geodesic if and only if it admits
  a non-vanishing section $\mbf{u}^b\in \Gamma(T\gamma(-2))$ such that
  $\Sigma^{\alpha\beta} = \bX^{\alpha\beta}_b\mbf{u}^b$
  is parallel along $\gamma$.
\end{prop}
\begin{proof} Using the identities \eqref{trids} we find
  \begin{equation}\label{magic}
    \mbf{u}^a\nabla_a \Sigma^{\alpha\beta} = \bX^{\alpha\beta}_b \mbf{u}^a\nabla_a \mbf{u}^b \, ,
  \end{equation}
  from which the result follows.
\end{proof}

\noindent We are now ready to prove the main Theorem  \ref{main-p} from the introduction:

\smallskip

\noindent{\em Proof of Theorem \ref{main-p}.}
The forward implication of the first statement is immediate from Proposition \ref{prop-proj-geod-par}.

For the converse let suppose that $\gamma$ is a curve and $\Sigma$ is
a non-vanishing $2$-tractor parallel along $\gamma$  and satisfying \eqref{main-p-eqn}. Since $X$ is nowhere
zero, we have $\Sigma =X\wedge V$ for some $V\in \Gamma( \cT (-1)|_\gamma )$.

Now choose (locally) a parametrisation of $\gamma$ with associated velocity $u$.
Differentiating  \eqref{main-p-eqn} along $\gamma$, then using \eqref{trids} and the fact that $\Sigma$ is parallel, we have
$$
  U\wedge \Sigma =0
$$
where $U^\beta:=u^b W^\beta_b$. Taking into account the weight of $V$,  it follows that locally along  $\gamma$
$$
  V^\alpha = \mbf{u}^b W^\alpha{}_b +\tau X^\alpha
$$ where $\mbf{u}\in \Gamma(T\gamma (-2))$ is a weighted velocity of
$\gamma$ and $\tau\in \Gamma(\ce(-2))$. Thus
$$
  \Sigma^{\alpha\beta}= \bX^{\alpha\beta}_b\mbf{u}^b .
$$
Then since $\Sigma $ is parallel along $\gamma$ it follows at once from \eqref{magic}
that $u^a\nabla_a \mbf{u}^b=0$, or equivalently $\mbf{u}^a\nabla_a \mbf{u}^b=0$.
So according to Lemma \ref{lem-wt-proj-geod}, $\gamma$ is an unparametrised geodesic.

Since $\Sigma$ satisfying \eqref{main-p-eqn} is necessarily of the form
$\Sigma^{\alpha\beta}=\bX^{\alpha\beta}_b\mbf{u}^b$ the uniqueness statement also follows.
\hfill $\Box$

\subsection{The model, incidence, and the double fibration}\label{model}

Recall that the standard homogeneous model for oriented projective
manifolds is the $n$-dimensional sphere arising as the ray
projectivisation $S^{n}:=\mathbb{P}_+(\mathbb{R}^{n+1})$ of
$\mathbb{R}^{n+1}$ (i.e.\ the double cover of $\mathbb{RP}^{n}$).

In this case it is well known \cite{BEG}
that the standard tractor bundle $\cT$ is naturally identified with the trivial bundle over $S^n$
\begin{equation}\label{pTc}
  T\mathbb{R}^{n+1}/\sim \to S^n
\end{equation}
where $\sim$ is the equivalence relation defined by $(p,v_p)\sim
  (q,\tilde{v}_q)$ iff $p,q$ belong to the same ray and $v_p$ and
$\tilde{v}_q$ are parallel with respect to the standard affine structure
on $\mathbb{R}^{n+1}$. Let us write $X^{\underline{\alpha}}$ for the
standard coordinates on $\mathbb{R}^{n+1}$, so that
$[X^{\underline{\alpha}}]$ denotes homogeneous coordinates on
$S^n$. Then it is easily verified that the Euler vector field
$X^{\underline{\alpha}}\partial/\partial X^{\underline{\alpha}} $ on $\mathbb{R}^{n+1}$ determines
a section $X^\alpha$ of $\cT(1)$ and that this is the (projective) canonical tractor on
$S^n$. Indeed it was these facts that inspired the projective tractor
notation in \cite{BEG}.

Choosing a section of $\mathbb{R}^{n+1}\to S^n$, for example the round
sphere, to represent the manifold $S^n$, it now follows immediately
from Theorem \ref{main-p} (and its proof) that on $S^n$, each geodesic
trace is the curve of points lying in the span of a 2-plane $\Sigma$
through the origin. That is, the geodesic traces are the great
circles. Of course this last fact is included in any understanding
that $\mathbb{P}_+(\mathbb{R}^{n+1})$ is a model for projective
geometry. Our main point here is that in this setting Theorem
\ref{main-p}, especially the incidence relation \eqref{main-p-eqn}, reduces precisely to the usual incidence relation
describing the great circles. Turning this around we see Theorem
\ref{main-p} generalises this incidence characterisation of great
circles to a general characterisation of geodesic traces.

In fact, in the homogeneous setting here, we can describe
the space of oriented geodesic traces of $S^n$ explicitly, as the $(2n-2)$-dimensional `twistor' space $\mbb{P}_+ \T$ of all oriented $2$-planes in $\R^{n+1}$ -- here $\mbb{T} := \{ \Sigma^{\alpha \beta} \in \Lambda^2 \R^{n+1} : \Sigma^{[\alpha \beta} \Sigma^{\gamma] \delta} = 0 \}$.

To relate geometric objects between $S^n$ and $\mbb{P}_+ \T$, we introduce
their correspondence space, that is the $(2n-1)$-dimensional
submanifold of $S^{n-1} \times \mbb{P}_+ \T$ defined by the incidence
relation $X \wedge\Sigma = 0$, for $[X] \in S^{n-1}$, $[ \Sigma ] \in
  \mbb{P}_+ \T$.  As a bundle over $S^n$, it is the ray projectivisation
$\mbb{P}_+(T S^n)$ of the tangent bundle of $S^n$, i.e.\ the fibre
$\mbb{P}_+(T_x S^n) \cong S^{n-1}$ at a point $x$ of $S^n$ is the set
of all oriented directions at $x$.  The correspondence space is
fibered over both $S^n$ and $\mbb{P}_+ \T$:
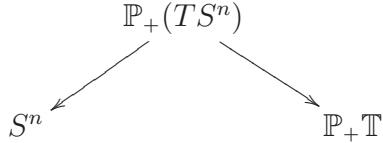
\begin{align*}
  \xymatrix{ & \mbb{P}_+ ( T S^n) \ar[dl] \ar[dr] &                \\
  S^n        &                                    & \mbb{P}_+ \T }
\end{align*}
A point in $\mbb{P}_+ \T$ gives rise to a geodesic in $S^n$ and conversely any geodesic in $S^n$ arises in this way. Indeed, a curve $\gamma$ in $S^n$ lifts canonically to a curve $\widetilde{\gamma}$ in $\mbb{P}_+ ( T S^n)$: the point of $\wt{\gamma}$ in the fibre over a point $\mr{x}$ of $\gamma$ is simply the oriented tangent direction $\dot{\gamma}$ at $\mr{x}$.  It is clear that $\gamma$ is a geodesic if and only if $\wt{\gamma}$ is (a subset of) a fibre of $\mbb{P}_+ ( T S^n)$ over $\mbb{P}_+ \T$.

In view of Proposition \ref{prop-proj-geod-par}, a point of  $\mbb{P}_+ \T$ gives rise to a simple section, up to a positive scale, of the tractor bundle $\wedge^2 \mc{T}$ over $S^n$, which is parallel with respect to the tractor connection, along \emph{any} direction, not only along the geodesic it defines.

\begin{rem}
  The spaces under considerations are generalized flag manifolds and admit a description as homogeneous spaces with automorphism group $G = \SL(\R^{n+1})$.
\end{rem}

\begin{rem}
  Replacing $\mbb{P}_+ ( T S^n)$ and $\mbb{P}_+ \T$ by $\mbb{P} ( T S^n)$ and $\mbb{P} \T$ respectively gives us a description of \emph{unoriented} great circles in $S^n$.
\end{rem}

\subsubsection{Initial conditions for geodesics}\label{sect-IC-geod}
The geodesic equation on a projective manifold $(M, \bp)$ is a
second-order (semi-linear) ODE.  Thus given a point $\mr{x}$ and a
vector $\mr{u}$ at that point, there is a unique parametrised (and
thus oriented) geodesic $\gamma: I \rightarrow M$ with $\gamma(0) =
  \mr{x}$ and velocity $\dot{\gamma} (0)= \mr{u}$. For a parametrisation
independent description, we merely need a direction at $\mr{x}$,
i.e.\ a point in the fibre $\mbb{P}_+(T_{\mr{x}} M)$.

In the homogeneous picture, a point $[\mr{\Sigma}]$ in $\mbb{P}_+ \T$
corresponds to an oriented great circle in $S^n$. But it also describes the
initial conditions at a point $\mr{x}$ of $S^n$ with homogeneous
coordinates $[\mr{X}]$.  Indeed,  $[\mr{\Sigma}]$ and $[\mr{X}]$
single out a point in the fibre $\mbb{P}_+(T_{\mr{x}} S^n)$, which is none
other than the direction of a geodesic (i.e. a great circle) at
$\mr{x}$.

Although there is no twistor space in the curved setting, this understanding still carries over.  Here, our initial conditions at a point $\mr{x}$ can simply be prescribed by a simple element of $\mbb{P}_+ (\Lambda^2 \mc{T}_{\mr{x}})$, which is isomorphic to $\mbb{P}_+\T$ as described above.

\begin{thm}\label{thm-IC-geod}
  Let $(M,\bp)$ be an $n$-dimensional projective manifold. Fix a point $\mr{x}$ in $M$ so that $X_{\mr{x}}$ is the canonical tractor based at $\mr{x}$. Then, for every non-zero element $\mr{\Sigma}$ of $\Lambda^2 \mc{T}_{\mr{x}}$ satisfying
  \begin{align}\label{eq-p-IC}
    X_{\mr{x}}\wedge \mr{\Sigma} =0 \, ,
  \end{align}
  locally  there exists a unique unparametrised geodesic $\gamma$ through $\mr{x}$. Further, the tractor $\Sigma$ associated to $\gamma$ (via Theorem \ref{main-p}) satisfies
  $\Sigma_{\mr{x}}= \lambda \mr{\Sigma}$ for some constant $\lambda>0$.  Any two non-zero elements of $\Lambda^2 \mc{T}_{\mr{x}}$ satisfying \eqref{eq-p-IC} give rise to the same oriented geodesic through $\mr{x}$ if and only if they differ by a positive constant multiple.
\end{thm}

\begin{proof}
  Let $\mr{x}$ and $\mr{\Sigma}$ be as in the statement of the
  theorem. The condition \eqref{eq-p-IC} tells us that $\mr{\Sigma}$
  is simple, and in particular, in a splitting of $\mc{T}_{\mr{x}}$,
  can be written in the form $\mr{\Sigma}^{AB} = 2 \mr{u}^b X^{[A}
    W^{B]}{}_b|_{\mr{x}}$ for some vector $\mr{u}^b$ in $T_{\mr{x}}
    M$. These are precisely the initial conditions which determine the
  unique local solution $I  \rightarrow \gamma$ to the geodesic equation
  \eqref{eq-aff-par-geod} through $\mr{x}$ for some $I \subset
    \R$. Clearly, if $\mr{\widetilde{\Sigma}} = c \, \mr{\Sigma}$ for
  some constant $c>0$, then $\mr{\widetilde{\Sigma}}$ yields
  $\mr{\widetilde{u}}=c \, \mr{u}$ in $T_{\mr{x}} M$, which determines
  the same oriented unparametrised geodesic.
\end{proof}

\section{Conformal geometry and distinguished curves}\label{c-sect}

On a given $n$-manifold $M$ ($n\geq 2$) two Riemannian or
pseudo-Riemannian metrics $g$ and $\wh{g}$ are said to be conformally
related if
\begin{align}\label{eq-conf-rel}
  \wh{g} & = \ee^{2 \, \phi} g \, , & \mbox{for some smooth function $\phi$.}
\end{align}
The Levi-Civita connections of the two metrics are related by the equation
\begin{align}\label{ltrans}
  \wh{\nabla}_a v^b & = \nabla_a v^b + \Upsilon_a v^b - v_a \Upsilon^b + \Upsilon_c v^c \delta^b{}_a \, , & \mbox{$v \in \Gamma(T M)$, }
\end{align}
where $\Upsilon_a := \nabla_a \phi$.

It follows easily from \eqref{ltrans} that in general the geodesics
of $g$ are not the same as the geodesics of $\wh{g}$, even after
possible reparametrisation. Null geodesics form the exception. A
parametrised geodesic $I\to \gamma$ is a {\em null geodesic} if it is
a geodesic and its velocity field $u^b$ is (at some point, equivalently
everywhere,) null. Thus we have $u^a\nabla_a u^b=0$ and $u_bu^b=0$ along the curve.
Thus from \eqref{ltrans} we deduce
\begin{equation}\label{ngf}
  u^a\wh{\nabla}_a u^b = 2(\Upsilon_au^a)u^b ,
\end{equation}
and we conclude that there is a reparametrisation of the curve with
velocity $\wh{u}^a$ satisfying $\wh{u}^a\nabla_a \wh{u}^b =0$. Thus
each null geodesic trace is conformally invariant.

As mentioned in the introduction the distinguished nowhere-null curves are
governed by the third-order conformal circle equation. To treat all
these things we review the basic tools of conformal geometry. We
follow the approach developed  in
\cite{BEG,CG-irred,G-PetersonCMP}. A useful summary review with some detail may be found in \cite{Curry-G-conformal}.

As in the case of projective geometry, the density bundles will be important for us.
Recall from Section \ref{proj-sect} that any manifold is
equipped with the oriented line bundle $\mathcal{K}:=(\Lambda^{n}
  TM)^2$. For conformal geometry it is convenient to adopt a notation for its
roots that differs slightly from that used above for projective
geometry: given $w\in \mathbb{R}$ we write
\begin{equation} \label{cdensities}
  \ce[w]:=\mathcal{K}^{\frac{w}{2n}} .
\end{equation}
For two metrics, conformally related as in \eqref{eq-conf-rel},
the induced  Levi-Civita connections
$\nabla$ and $\wh{\nabla}$
on the bundles $\cE[ w ]$
are related by
\begin{equation}\label{conformal_change_on_densities}
  \widehat{\nabla}_a\sigma=\nabla_a\sigma+ w \Upsilon_a \sigma \, .
\end{equation}

\subsection{Conformal geometry and conformal tractor calculus}\label{csec}

\newcommand{\bg}{\mathbf{g}}
\newcommand{\cc}{\boldsymbol{c}}
A \emph{conformal structure} on a smooth manifold $M$ is an
equivalence class $\bc$ of metrics, whereby two metrics $g,
  \wh{g}\in\bc$ are conformally related as in \eqref{eq-conf-rel}.
A conformal manifold is equipped with a canonical section $\mbf{g}$ of $\odot^2T^*M[2]$ called the {\em conformal metric} (see e.g. \cite{Curry-G-conformal}).
This is non-degenerate (as a metric on $TM[-1]$) and carries the information of the conformal structure. Indeed any choice of metric $g\in \bc$ is equivalent to a choice of {\em scale} $\si\in \Gamma (\ce_+[1])$ by the formula
$$
  g=\si^{-2} \mbf{g} \, .
$$
The conformal structure is said to be of signature $(p,q)$ if this is the signature of $\mbf{g}$ (equivalently the signature of any $g\in \bc$).

For Riemannian and pseudo-Riemannian geometry we describe another decomposition of
the curvature. The curvature tensor $R_{ab}{}^{c}{}_{d}$ of the Levi-Civita connection $\nabla_a$ of a metric  splits as
\begin{align*}
  R_{abcd} & = W_{abcd} + 2 \, \mbf{g}_{c[a} \Rho_{b]d} - 2 \, \mbf{g}_{d[a} \Rho_{b]c} \, ,
\end{align*}
where $W_{ab}{}^{c}{}_{d}$ is the conformal Weyl tensor and $\Rho_{ab}$
is the conformal Schouten tensor. The conformal Weyl tensor is totally
tracefree, satisfies the algebraic Bianchi identities, and is
conformally invariant. The conformal Cotton tensor is defined by
$Y_{abc} := 2 \, \nabla_{[b} \Rho_{c]a}$. The Bianchi identity yields
$Y_{cab} = (n-3) \nabla_d W_{ab}{}^d{}_c$.

As for projective geometry, on a general conformal manifold there is
no distinguished connection on $TM$ but there is one on a related higher
tractor bundle that we again denote $\cT$. In this case one considers the bundle $J^2\ce[1]$ of 2-jets
of the $\ce[1]$. By definition one has the jet exact sequence at 2-jets
\begin{equation}\label{J2}
  0\to \odot^2T^*M[1]\to J^2\ce[1]\to J^1\ce[1]\to 0.
\end{equation}
but on a conformal manifold we may split $\odot^2 T^*M$, and hence also
$\odot^2T^*M[1]$ using $\mbf{g}$. We have $\odot^2T^* M[1]=\ce[-1] \oplus \odot^2_0T^*
  M[1] $ where $\odot^2_0T^* M[1]$ is the metric trace-free part of $\odot^2T^*
  M[1]$, $\ce[-1]\to \odot^2T^* M[1]$ is included by $\rho\mapsto \rho \mbf{g}$
and this is split by taking $1/n$ times the $\mbf{g}$-trace. The standard
tractor bundle $\cT$ (or we write $\ce^A$ in the abstract index
notation) is defined to be quotient of $J^2\ce[1]$ by the image of the
map $\odot^2_0T^*M[1]\to J^2\ce[1]$. It then follows from \eqref{J2} that
$\cT$ has the composition structure
\begin{equation}\label{tseqs}
  0\to \ce[-1]\stackrel{X}{\to} \cT\to J^1\ce[1]\to 0 \quad\mbox{and} \quad 0\to T^*M[1]\to  J^1\ce[1]\to \ce[1]\to 0
\end{equation}
which we summarise as $\ce^A=\ce[1]\lpl \ce^a[1]\lpl\ce[-1] $. The
mapping $X$ may be viewed as a section of $\ce^A[1]$ and this is the
conformal {\em canonical} tractor. It turns out that there is
an invariant tractor metric $h=h_{AB}$, of signature $(p+1,q+1)$, and
with respect to this, $X$ is null, $h(X,X)=0$. Note that this tractor metric identifies $\mc{T}$ with the dual tractor bundle $\mc{T}^*$.

Given a choice of metric $g\in \bc$ the sequences \eqref{tseqs} split, as
discussed in e.g.\ \cite{CG-irred,Curry-G-conformal}, so that
$\ce^A\stackrel{g}{=}\ce[1]\oplus \ce^a[1]\oplus\ce[-1] $, and an
element $V^A$ of $\ce^A$ may be represented by a triple
$(\si,\mu_a,\rho)$, or equivalently by
\begin{equation}\label{Vsplit}
  V^A=\si Y^A+\mu^aZ^A {}_a+\rho X^A,
\end{equation}
where we have raised indices using $h^{-1}=h^{AB}$ and
$\mbf{g}^{-1}=\mbf{g}^{ab}$. The last display defines the algebraic splitting
operators $Y:\ce[1]\to \cT$ and $Z :TM[-1]\to \cT$ (determined by the
choice $g\in \bc$) which may be viewed as sections $Y^A\in
  \Gamma(\ce^A[-1])$ and $Z^A{}_a\in \Gamma(\ce^A_a[1])$. As a quadratic
form the tractor metric is given by
$$
  V^A\mapsto 2\si \rho + \mbf{g}_{ab}\mu^a\mu^b ,
$$
in terms of the splitting, or equivalently its inverse is
$$
  h^{AB}= 2X^{(A}Y^{B)}+\mbf{g}^{ab}Z^A{}_aZ^B{}_b.
$$
We then use the tractor metric to raise and lower tractor indices
so we also have $h_{AB}= 2X_{(A}Y_{B)}+\mbf{g}_{ab}Z_A{}^aZ_B{}^b$, with
$X^AY_A=1$, $Z_A{}^aZ^A{}_b=\delta^{a}{}_b $ and all other pairings of the
splitting operators giving a zero section.

While $X^{A}$ is conformally invariant, a change of tractor splitting given
by \eqref{eq-conf-rel} is equivalent to the transformations
\begin{align*}
  \wh{Z}^{A}{}_a & = Z^{A}{}_a - \Upsilon_a X^{A} \, , & \wh{Y}^{A} & = Y^{A} + \Upsilon^a Z^{A}{}_a + \frac{1}{2} \Upsilon^a \Upsilon_a X^{A}\,
\end{align*}
where, as usual, $\Upsilon_a=\nabla_a \phi$.
Thus  we conclude, for example, that there is a conformally invariant injective
bundle map from $TM[-2]$ into $\Lambda^2\cT$:
\begin{equation}\label{injpc}
  \bX^{AB}_b: \ce^b[-2]\to \ce^{[AB]} \quad\mbox{given by} \quad \mbf{v}^b\mapsto 2X^{[A} Z^{B]}{}_b \mbf{v}^b.
\end{equation}

There is a canonical conformally invariant (normal) tractor connection on $\cT$ that preserves $h$ that we shall
also denote $\nabla_a$.
It can be coupled
to the Levi-Civita connection of any metric in $\bc$, and its
action on the splitting operators
is then given by
\begin{align}\label{ctrids}
  \nabla_a X^{A} & =  Z^{A}{}_a \, , & \nabla_a Z^{A}{}_b & = - \Rho_{ab} X^{A} - \mbf{g}_{ab} Y^{A} \, , &
  \nabla_a Y^{A} & = \Rho_a{}^{b} Z^{A}{}_b \, .
\end{align}
The general action on a section of a tractor bundle can then be deduced from the Leibniz rule.

Finally, the \emph{tractor curvature} $\Omega_{ab}{}^{C}{}_{D}
$ of the tractor connection, defined by $\Omega_{ab}{}^{{C}}{}_{{D}} \Phi^{D} := 2 \, \nabla_{[a} \nabla_{b]}
  \Phi^{C}$, for any $\Phi^A \in \Gamma( \mc{T} )$, can be expressed in a splitting as
\begin{align}\label{ctract-curv}
  \Omega_{ab}{}_{C D} & = W_{abcd} Z_{C}{}^c Z_{D}{}^d -2 Y_{cab} X^{\phantom{c}}_{[C} Z_{D]}{}^c
\end{align}
A conformal structure is said to be \emph{(locally) flat} if this
tractor curvature vanishes as this happens if and only if there is
locally a flat metric in the conformal class.

\subsection{Null geodesics}\label{null-sec}
In view of the observation \eqref{ngf}, some conformal aspects of null
geodesics correspond to projective features of arbitrary
geodesics. The following Lemma gives a parametrisation independent equation for null
geodesics and is analogous to Lemma
\ref{lem-wt-proj-geod}. Here we write $T \gamma [w]$ for $T\gamma \otimes \mc{E}[w]|_\gamma$ for any weight $w \in \R$.
\begin{lem}\label{lem-wt-null-geod}
  Let $\gamma$ be an oriented curve on $(M,\bc)$.  Then $\gamma$ is a unparametrised null
  geodesic, with respect to $\bc$, if and only if there exists a
  non-vanishing null vector
  field
  $\mbf{u}^a\in \Gamma(T\gamma [-2])$ along
  $\gamma$ satisfying the conformally invariant equation
  \begin{align}\label{eq-proj-inv-conf}
    \mbf{u}^b \nabla_b \mbf{u}^a & = 0 \, , & \mbox{for any $g \in \bc$ with Levi-Civita $\nabla$.}
  \end{align}
  The weighted velocity field $\mbf{u}$ is unique up to a positive factor that is
  constant along $\gamma$.
\end{lem}
\begin{proof}
  The conformal invariance of \eqref{eq-proj-inv-conf} is an easy consequence of
  \eqref{ltrans}, \eqref{conformal_change_on_densities}.

  The remainder of the proof is then a simple adaption of the proof of
  Lemma \ref{lem-wt-proj-geod} that uses now \eqref{ngf}.
\end{proof}

We now find the tractor picture also agrees with the projective case as follows. Given a curve $\gamma$ and
$\mbf{u}^a\in \Gamma(T\gamma [-2])$ we may use \eqref{injpc} to form the conformally
invariant weight zero $2$-tractor field
\begin{align}\label{eq-geod-bitract-null}
  \Sigma^{{A}{B}} & := \bX^{AB}_b \mbf{u}^b ,
\end{align}
along $\gamma$.  By construction, $\Sigma$ is simple, i.e.\ $\Sigma
  \wedge \Sigma = 0$. Furthermore, assuming $\mbf{u}^a$ is
non-vanishing, it is easily verified that, the curve $\gamma$ is null
if and only if $\Sigma$ is \emph{totally null}, i.e.\ the span of $\Sigma$ is totally null. This can be characterised equivalently by the property that
\begin{align}\label{nil-eq}
  \Sigma^A{}_B\Sigma^B{}_C=0 \, ,
\end{align}
that is, $\Sigma^A{}_B$ is nilpotent of order $2$.
We then have a conformal analogue
of Proposition \ref{prop-proj-geod-par}.

\begin{prop}\label{nullprop}
  An oriented curve $\gamma$ on $({M},\bc)$ is an unparametrised null geodesic if and only if it admits a non-vanishing section $\mbf{u}^b\in \Gamma(T\gamma [-2])$ such that
  $\Sigma^{\alpha\beta} = \bX^{\alpha\beta}_b\mbf{u}^b$
  is parallel along $\gamma$.
\end{prop}
\begin{proof}
  It follows from the identities \eqref{ctrids} that
  \begin{equation}\label{cmagic}
    \mbf{u}^a\nabla_a \Sigma^{AB} =\bX^{AB}_b \mbf{u}^a\nabla_a\mbf{u}^b
    -2X^{[A}Y^{B]}\mbf{g}_{ab}\mbf{u}^a\mbf{u}^b .
  \end{equation}
  At each point the terms on the right hand side are linearly
  independent if non-zero.
\end{proof}

We are now in position to prove Theorem \ref{main-nullc} in analogy to the projective case.

\noindent{\em Proof of Theorem \ref{main-nullc}.}
The argument follows Proof of Theorem \ref{main-p}, but now uses
\eqref{cmagic} and the conformal tractor calculus.
\hfill $\Box$

It follows immediately from the proof that any $2$-tractor $\Sigma$ satisfying \eqref{main-nullc-eqn} must be of the form \eqref{eq-geod-bitract-null}, in particular, simple and totally null.

\subsection{The model and incidence}\label{nmodel}

\newcommand{\cQ}{\mathcal{Q}}
To obtain the standard homogeneous model for oriented conformal
structures of signature $(p,q)$ we begin with $\mathbb{R}^{n+2}$,
equipped with a fixed symmetric non-degenerate bilinear form $h$, of
signature $(p+1,q+1)$.
In the case of definite signature, the model is discussed in detail in \cite{Curry-G-conformal} and \cite{GRiemSig}.
We then form
$S^{n+1}:=\mathbb{P}_+(\mathbb{R}^{n+2})$ (cf.\ Section \ref{model}) and
consider the quadric $\cQ:=\mathbb{P}_+\mathbb{N}$, where
$\mathbb{N}:=\{X\in\mathbb{R}^{n+2} \mid h(X,X) =0\}$ is the null quadric in $\mathbb{R}^{n+2}$.
As a smooth manifold this is a sphere product
$S^p \times S^q$ smoothly embedded as codimension 1 submanifold in
$\mathbb{P}_+(\mathbb{R}^{n+2})$.

There is the projective standard tractor bundle $\cT$ on
$\mathbb{P}_+(\mathbb{R}^{n+2})$ as constructed in Section \ref{model}
(see especially \eqref{pTc}) and the conformal standard  tractor bundle
$\cT_c$ is simply the pull back of $\cT$ along the embedding $i:\cQ\to
  \mathbb{P}_+(\mathbb{R}^{n+2})$. Similarly, in this setting the
conformal tractor connection $\nabla^{\cT_c}$ is simply the
restriction of the projective tractor connection, and so arises from the
parallel transport of the standard affine structure on the vector
space $\mathbb{R}^{n+2}$. The canonical tractor
$X\in \cT_c[1]$ is just the pullback to $\cQ$ of the conformal
canonical tractor, that is, it is the tautological section of $\cT_c[1]$
determined at each point $x\in \cQ$ by the homogeneous coordinates for
$x$. The conformal tractor metric then arises in the obvious way from
the symmetric bilinear form $h$. We subsequently drop the subscript
and denote the conformal standard tractor bundle by $\cT$.

Arguing as in the projective case, it now follows immediately
from these facts and Theorem \ref{main-nullc} that on $\cQ$, each unparametrised null geodesic is the curve of points lying in the span of a totally null 2-plane
through the origin. So Theorem \ref{main-nullc} generalises this
incidence relation to a general characterisation of null geodesic traces.

The double fibration picture is very similar to the one
presented in Section \ref{model}. Here, the space of all oriented null geodesics in $\mc{Q}$ is the $(2n-3)$-dimensional ray projectivisation $\mbb{P}_+ \T$ of
\begin{align*}
  \T := \left\{ \Sigma^{{A} {B}} \in  \Lambda^2 \R^{n+2} : \Sigma^{[{A} {B}} \Sigma^{{C}] {D}} = 0 \, , \Sigma^A {}_{C} \Sigma^C {}_{B} = 0 \right\} \, .
\end{align*}
The space $\mbb{P}_+ \T$ consists of all oriented linear $1$-dimensional subspaces contained in $\mc{Q}$.

The correspondence space between $\mc{Q}$ and $\mbb{P}_+ \T$ is the $(2n-2)$-dimensional submanifold of $\mc{Q} \times \mbb{P}_+ \T$ defined by the incidence relation $X \wedge\Sigma = 0$, for $[X] \in \mc{Q}$, $[ \Sigma ] \in \mbb{P}_+ \T$. As a bundle over $\mc{Q}$, it is the ray projectivisation $\mbb{P}_+ \mc{C}$ of the bundle $\mc{C}$ of null cones over $\mc{Q}$: a fibre of $\mbb{P}_+ \mc{C}$ over a point $x$ of $\mc{Q}$ consists of all oriented null directions through $x$.
Again, the correspondence space yields the double fibration:
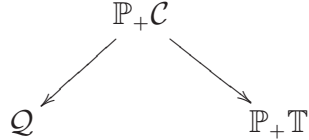
\begin{align*}
  \xymatrix{ & \mbb{P}_+ \mc{C} \ar[dl] \ar[dr] &                \\
  \mc{Q}     &                                  & \mbb{P}_+ \T }
\end{align*}
Analogously to the projective case, the tangent directions of a curve $\gamma$ in $\mc{Q} \cong S^n$ canonically determine a lift of $\gamma$ to $\mbb{P}_+ \mc{C}$, which descends to a point in $\mbb{P}_+ \T$ if and only if $\gamma$ is a null geodesic.

\begin{rem}
  These spaces are generalized flag manifolds and admit a description as homogeneous spaces with automorphism group $G = \SO(p+1,q+1)$.
\end{rem}

\begin{rem}
  Note that replacing $\mbb{P}_+ \mc{C}$ and $\mbb{P}_+ \T$ by $\mbb{P} \mc{C}$ and $\mbb{P} \T$ respectively yields a double fibration for unoriented null geodesics.
\end{rem}

\subsubsection{Initial conditions for null geodesics}\label{sect-IC-n-geod}
The discussion above allows us to treat the problem of initial conditions for null geodesics on a conformal manifold in the same way as in Section \ref{sect-IC-geod}. We leave the details to the reader. One obtains the following theorem:

\begin{thm}\label{thm-IC-n-geod}
  Let $(M,\bc)$ be a $n$-dimensional conformal manifold of indefinite signature. Fix a point $\mr{x}$ in $M$ so that $X_{\mr{x}}$ is the canonical tractor based at $\mr{x}$. Then, for every totally-null non-zero element $\mr{\Sigma}$ of $\Lambda^2 \mc{T}_{\mr{x}}$ satisfying
  \begin{align}\label{eq-cn-IC}
    X_{\mr{x}}\wedge \mr{\Sigma} =0 \, ,
  \end{align}
  there locally exists a unique unparametrised null geodesic $\gamma$ through $\mr{x}$.  Further, the tractor $\Sigma$ associated to $\gamma$ (via Theorem \ref{main-nullc}) satisfies
  $\Sigma_{\mr{x}}= \lambda \mr{\Sigma}$ for some constant $\lambda>0$. Any two totally null elements of $\Lambda^2 \mc{T}_{\mr{x}}$ that satisfy \eqref{eq-cn-IC} give rise to the same oriented geodesic through $\mr{x}$ if and only if they differ by a positive constant multiple.
\end{thm}

\begin{proof}
  The proof proceeds in the same way as that of Theorem \ref{thm-IC-geod}. The only difference is the additional requirement that $\mr{\Sigma}$ is totally null, which does not follow from \eqref{eq-cn-IC}. This condition ensures that the vector $\mr{u}$ tangent at $\mr{x}$ is null.
\end{proof}

\subsection{Conformal circles}\label{Conf-circ-sect}

Next, we consider the conformally invariant generalisation of
nowhere-null geodesics known as \emph{conformal circles} or
\emph{conformal geodesics} (see \cite{Bailey1990a,BEG,Tod} and
references therein). We shall follow the definition of conformal
circles introduced in \cite{BEG}, which is expressed in the language
of tractor calculus.

\subsubsection{Nowhere-null curves}

On a Riemannian, pseudo-Riemannian, or conformal manifold, by a {\em
    nowhere-null} curve we mean a curve $\gamma$ that admits a velocity
field $u\in \Gamma(T\gamma)$ with $u^au_a=\mbf{g}(u,u)$ nowhere zero.
Let us first highlight a number of conformal properties of nowhere-null curves.

Note that such a velocity
field $u^a$ determines a scale $\si_u\in \Gamma(\mc{E}_+[1]|_\gamma)$ along the curve via the
definitions
\begin{equation}\label{usc}
  \sigma_u := \left\{ \begin{array}{ll}
    \sqrt{u^a u_a}  & \quad \mbox{if $u^a$ is spacelike,} \\
    \sqrt{-u^a u_a} & \quad \mbox{if $u^a$ is timelike.}
  \end{array}
  \right.
\end{equation}

\begin{lem}\label{lem-wt-unit-vel}
  Let $\gamma$ be an oriented nowhere-null curve. There exists a unique
  weighted vector field $\mbf{u}^a \in \Gamma(T \gamma[-1])$ along $\gamma$
  that is compatible with the orientation and  satisfies
  \begin{align}\label{eq-unit-wt-vel}
    \mbf{u}^a \mbf{u}_a & =
    \begin{cases}
      1 \, ,  & \mbox{if $\gamma$ is spacelike,} \\
      -1 \, , & \mbox{if $\gamma$ is timelike.}
    \end{cases}
  \end{align}
\end{lem}

\begin{proof}
  Given an oriented curve $\gamma$, let $u^a$ be a nowhere zero velocity
  vector field compatible with the orientation, and define $\sigma_u$ by \eqref{usc}.
  Then $\mbf{u}^a
    := \sigma_u^{-1} u^a$ does not depend on the choice of velocity vector
  $u^a$, and satisfies the properties given by \eqref{eq-unit-wt-vel}.
\end{proof}

\begin{defn}
  Let $\gamma$ be an oriented nowhere-null  curve on $(M,\bc)$. We
  call the canonical weighted vector field $\mbf{u}^a$ given in Lemma
  \ref{lem-wt-unit-vel}, the \emph{weighted velocity} of
  $\gamma$. Given $g\in \bc$, we define the \emph{weighted acceleration} of
  $\gamma$ to be
  \begin{align}\label{eq-wt-acc}
    \mbf{a}^b & := \mbf{u}^c \nabla_c \mbf{u}^b \in \Gamma(\ce^b[-2]|_\gamma) \, .
  \end{align}
\end{defn}
It is clear that
\begin{align}\label{eq-rel-wt-vel-acc}
  \mbf{u}^b \mbf{a}_b = 0 \, .
\end{align}
Furthermore if $\wh{g}=e^{2\phi}g$ then  we have
$\widehat{\mbf{a}}^b = \mbf{a}^b + \mbf{u}^a \Upsilon_a \mbf{u}^b \mp
  \Upsilon^b$ whenever  $\mbf{u}^a\mbf{u}_a=\pm 1$.

Note that given an oriented curve $\gamma$, any scale $\si\in \Gamma(\ce_+[1]|_\gamma)$ along the curve
determines a particular velocity vector
$u^a := \sigma \mbf{u}^a$ with $\sigma^2 = |u^a u_a|$, and thus a
parametrisation of $\gamma$ (up to additive constant). We note further
the relation between the weighted acceleration and the acceleration
vector $a^b := u^c \nabla_c u^b$ resulting from such a choice:
\begin{align}\label{eq-wt-uwt-acc}
  a^c & = \sigma^2 \mbf{a}^c + \sigma \left( \mbf{u}^b \nabla_b \sigma \right) \mbf{u}^c \, , & \mbf{a}^c & = \sigma^{-2} a^c - \sigma^{-3} \left( u^b \nabla_b \sigma \right) u^c \, .
\end{align}

\subsubsection{The conformal circle equations}\label{ccirc}

All curves locally admit  {\em projective parametrisation} as
defined in \cite{Bailey1990a,BEG}.  A parametrised curve $\gamma$ is a
projectively parametrised conformal circle if and only if it satisfies
the equation \eqref{par-conf-circ}.

We are interested in parametrisation independent descriptions of
conformal circles. As a step toward this we use the following easily verified result from  \cite{Bailey1990a}.
\begin{prop}
  Let $\gamma$ be any curve on $(M,\bc)$ with velocity
  vector field $u^a\in\Gamma(T\gamma)$ and acceleration $a^b := u^c \nabla_c u^b$.  The equations
  \begin{align}\label{eq-conf-circ}
    \left( u^b \nabla_b a^{[a} \right) u^{b]} & =  3 \frac{u \cdot a}{u \cdot u} a^{[a} u^{b]} + ( u \cdot u ) u^c \Rho_c{}^{[a} u^{b]}
  \end{align}
  are parametrisation independent. Moreover the  equation is satisfied if and only if there is a reparametrisation that obeys the conformal circle equation \eqref{par-conf-circ}.

\end{prop}

The analogue of Lemmata \ref{lem-wt-proj-geod} and
\ref{lem-wt-null-geod} is given by:
\begin{lem}
  Let $\gamma$ be an oriented nowhere-null curve on $(M,\bc)$. Then $\gamma$ is a conformal circle  if and only if its weighted  velocity $\mbf{u}^a$   and acceleration $\mbf{a}^a$ satisfy the conformally invariant equation
  \begin{align}\label{eq-proj-inv-wt}
    \left( \mbf{u}^c \nabla_c \mbf{a}^{[a} \right) \mbf{u}^{b]} & = \pm \mbf{u}^c \Rho_c{}^{[a} \mbf{u}^{b]} \, , & \mbox{whenever $\mbf{u}^a\mbf{u}_a=\pm 1$,}
  \end{align}
  or equivalently,
  \begin{align}\label{eq-proj-inv-wt2}
    \mbf{u}^b \nabla_b \mbf{a}^a & = \pm \mbf{u}^b \Rho_b {}^a  - ( \Rho_{bc} \mbf{u}^b \mbf{u}^c \pm \mbf{a} \cdot \mbf{a} ) \mbf{u}^a \, , & \mbox{whenever $\mbf{u}^a\mbf{u}_a=\pm 1$,}
  \end{align}
  for any $g \in \bc$ with Levi-Civita connection $\nabla$.
\end{lem}
\begin{proof}
  We first note that the equivalence of \eqref{eq-proj-inv-wt} and
  \eqref{eq-proj-inv-wt2} follows from \eqref{eq-rel-wt-vel-acc}. The
  rest of the lemma can be proved simply by choosing a density $\sigma
    \in \Gamma(\mc{E}_+[1]|_\gamma)$ along $\gamma$, setting $u^a := \sigma \mbf{u}^a$ so that $u^a$ is
  a vector field tangent to $\gamma$ with $u^a u_a = \pm \sigma^2$, using
  \eqref{eq-wt-uwt-acc}, and substituting \eqref{eq-proj-inv-wt} to
  get \eqref{eq-conf-circ}.
\end{proof}

\begin{rem} Equation \ref{eq-proj-inv-wt2} recovers Tod's \cite[Equations (11) and (14)]{Tod}, see also \cite[Equation
    (3)]{unique-MikE}. In a private communication, Michael Eastwood
  has informed us that in a forthcoming work he and Lenka Zalabov\'a
  provide a tractor derivation of the equation \eqref{eq-proj-inv-wt2}
  via Proposition 1 of the Doubrov-\v Z\'adn\'\i k article \cite{Doub-Zad}.

  Since the equation \nn{eq-proj-inv-wt2} is scale (equivalently,
  parametrisation) independent we can easily use it to deduce an equation
  for any choice of scale. For example, let $\sigma \in \mc{E}_+[1]$ be
  a scale determining a metric $g_{ab} := \sigma^{-2}
    \mbf{g}_{ab}$. If $\mbf{u}^a$ is the weighted velocity of a curve
  $\gamma$ with $\mbf{u}^a \mbf{u}_a = \pm 1$, then $u^a := \sigma
    \mbf{u}^a$ is unit in the sense that $u^a u^b g_{ab} = \pm 1$. If $\gamma$ is a conformal circle, then  we conclude at once that its acceleration
  vector field satisfies
  \begin{align}\label{eq-par-conf-circ-unit}
    u^b \nabla_b a^a & = \pm u^b \Rho_b{}^a - \left( \Rho_{bc} u^b u^c \pm (a \cdot a) \right) u^a \, .
  \end{align}
  This generalises to arbitrary signature the definition of conformal circles presented in \cite{Yano1938,Yano1952}.
\end{rem}

\subsubsection{The tractor formulation} \label{trac-form}
\begin{defn}
  Let $\gamma$ be an oriented nowhere-null curve on a conformal manifold $({M} ,
    \bc)$, with weighted velocity $\mbf{u}^a$.  Choose $\sigma \in \Gamma(\mc{E}_+[1]|_\gamma)$.
  The
  \emph{velocity tractor} and the \emph{acceleration tractor}
  associated to $\gamma$ and $\si$ are
  defined to be
  \begin{align*}
    U^{A} & := \sigma \mbf{u}^a \nabla_a \left( \sigma^{-1} X^{A} \right) \, , &
    A^{A} & := \sigma \mbf{u}^a \nabla_a U^{A} \, ,
  \end{align*}
  respectively.
\end{defn}
When written out explicitly, we have, for any $g \in \bc$ with Levi-Civita connection $\nabla$,
\begin{align*}
    & U^{A} = \left( - \sigma^{-1} \mbf{u}^a \nabla_a \sigma \right) X^{A} + \mbf{u}^a Z_a^{A} \, ,                                                                                                         \\
    & A^{A} = \left( \sigma^{-1} \left( \mbf{u}^a \nabla_a \sigma \right)^2 - \mbf{a}^a \nabla_a \sigma - \mbf{u}^a \mbf{u}^b \nabla_a \nabla_b \sigma - \mbf{u}^a \mbf{u}^b \Rho_{ab} \sigma \right) X^{A} \\
    & \qquad \qquad \qquad\qquad \qquad \qquad + \left( \mbf{a}^a \sigma - \left( \mbf{u}^b \nabla_b \sigma \right) \mbf{u}^a \right) Z_a^{A}+ (\mp \sigma) Y^{A} \, ,
\end{align*}
whenever $\mbf{u}^a \mbf{u}_a = \pm 1$, and with $\mbf{a}^b := \mbf{u}^a \nabla_a \mbf{u}^b$.

Note that both $U^{A}$ and $A^{A}$ depend on the choice of scale $\si$ along $\gamma$, and thus on a choice of parametrisation, see Section \ref{param} below. However, the $3$-tractor defined by
\begin{equation}\label{eq-SigmaXUA}
  \Sigma^{{A} {B} {C}} :=  \sigma^{-1} \,  6 \, X^{[{A}} U^{{B}} A^{{C}]}
\end{equation}
is independent of $\si$. We have the following result.
\begin{lem}\label{tri-lem}
  An unparametrised oriented nowhere-null curve $\gamma$ determines a canonical $3$-tractor $\Sigma\in \Gamma(\Lambda^3\cT|_\gamma)$ along it via  \eqref{eq-SigmaXUA}.
\end{lem}
\begin{proof}
  Along $\gamma$ choose a scale
  $\si\in \Gamma(\ce_+[1]|_\gamma)$. Then by the conformal invariance of the tractor connection it follows that the tractors
  $\sigma^{-1} X^{A}$, $U^{A}$ and $A^{A}$ depend only on $\gamma$ and $\si$. Thus $\Sigma$ as in \eqref{eq-SigmaXUA} can depend only on $\gamma$ and $\si$.

  To facilitate calculation we pick some other background scale
  to split the tractor bundles. Then
  $\Sigma$, as defined by \eqref{eq-SigmaXUA}, is given by
  \begin{align} \label{cSigma}
    \Sigma^{{A} {B} {C}} & =  \pm  6 \, \mbf{u}^c \, X^{[{A}} Y^{{B}} Z^{{C}]}{}_c + 6 \, \mbf{u}^b \mbf{a}^c \, X^{[{A}}  Z^{{B}}{}_b Z^{{C}]}{}_c \, , & \mbox{whenever $\mbf{u}^a \mbf{u}_a = \pm 1$,}
  \end{align}
  where we note that there is no dependency on $\si$.
\end{proof}

\begin{prop}\label{prop-conf-circ-par}
  Let $\gamma$ be a nowhere-null oriented curve on $({M}, \bc)$ with
  associated $3$-tractor $\Sigma^{{A}{B}{C}}$ as defined by \nn{eq-SigmaXUA}. Then $\gamma$ is
  a conformal circle if and only if $\Sigma^{{A}{B}{C}}$ is
  constant along $\gamma$.
\end{prop}

\begin{proof} Let $\mbf{u}^a$ be the weighted velocity of $\gamma$. Then differentiating $\Sigma^{{A}{B}{C}}$ gives
  \begin{align*}
    \mbf{u}^d \nabla_d \Sigma^{{A} {B} {C}} & =  6 \left( \mbf{u}^d \nabla_d \mbf{a}^c \mp \mbf{u}^d \Rho_d{}^c \right) \mbf{u}^b \, X^{[{A}}  Z^{{B}}{}_b Z^{{C}]} {}_c\, , & \mbox{whenever $\mbf{u}^a \mbf{u}_a = \pm 1$,}
  \end{align*}
  and the result follows immediately -- see equation \eqref{eq-proj-inv-wt}.
\end{proof}

Before we continue we need a few more technical results, as
follows. Given an oriented curve $\gamma$, and a choice of scale $\si$ along $\gamma$, these are established
using the definition of $U$, $A$, and the tractor identities following  \nn{Vsplit} in Section \ref{csec}:
\begin{align}
  U \cdot U & = \pm 1 \, ,
  \label{eq-U.U}\\
  A \cdot A & = \pm 2 \, \sigma \left(  \mbf{u}^a \mbf{u}^b \nabla_a \nabla_b \sigma + \mbf{a}^a \nabla_a \sigma - \frac{1}{2} \sigma^{-1} (\mbf{u}^a \nabla_a \sigma)^2+ \mbf{u}^a \mbf{u}^b \Rho_{ab} \sigma \pm \frac{1}{2} \mbf{a}^b \mbf{a}_b \sigma \right) \, , \label{eq-A.A} \\
  X \cdot A & = \mp \sigma \, , \label{eq-X.A}                                                                                                                                                                                                                                        \\
  X \cdot U & = 0 \, , \label{eq-X.U}                                                                                                                                                                                                                                                 \\
  U \cdot A & = 0 \, ,\label{eq-U.A}
\end{align}
where the weighted velocity $\mbf{u}^a$ satisfies $\mbf{u}^a \mbf{u}_a = \pm 1$, and $\mbf{a}^b$ is the weighted acceleration.

The \emph{signature} of a simple $k$-tractor $\Sigma$ will refer to the signature of the restriction of the tractor metric to the span of the factors of $\Sigma$. The identities above yield the following lemma.
\begin{lem} \label{llem}
  Let $\gamma$ be an oriented nowhere-null curve on $({M}, \bc)$ with weighted velocity $\mbf{u}^a$ and
  associated $3$-tractor $\Sigma$ as defined by \nn{eq-SigmaXUA}.  At any point $\Sigma$ has signature $(+,+,-)$ if $\gamma$ is spacelike, and $(+,-,-)$ if $\gamma$ is timelike. Moreover, we have $|\Sigma|^2 := \frac{1}{6} \Sigma^{ABC}\Sigma_{ABC}=\mp 1$ whenever $\mbf{u}^a \mbf{u}_a = \pm 1$.
\end{lem}

\begin{proof}
  Let us define
  \begin{align*}
    B^{A} & := \frac{1}{\sqrt{2}} \left( A^{A} \pm \left( \frac{A \cdot A}{2} - 1 \right) \sigma^{-1} X^{A} \right) \, , & C^{A} & := \frac{1}{\sqrt{2}} \left( A^{A} \pm \left( \frac{A \cdot A}{2} + 1 \right) \sigma^{-1} X^{A} \right) \, ,
  \end{align*}
  where the sign is chosen according to whether $\gamma$ is spacelike or
  timelike respectively. Then, from \eqref{eq-X.A}, it is easy to check that $B \cdot
    B = 1$ and $C \cdot C = -1$, and that $U^{A}$, $B^A$ and $C^A$ are
  mutually orthogonal. Using \eqref{eq-U.U} now yields the
  signature of $\Sigma$. The final claim also follows easily, or directly from the
  identities above.
\end{proof}

\medskip

We are  now ready to prove the main result.\\
\noindent{\em Proof of Theorem \ref{main-c}.}
If $\gamma$ is a conformal circle then we can simply take
$\Sigma^{{A}{B}{C}}$ to be \eqref{eq-SigmaXUA} and so the forward direction follows from Proposition \ref{prop-conf-circ-par}.

For the
converse suppose that $\gamma $ is a nowhere-null oriented curve satisfying
\begin{align}\label{eq-inc-conf-circ}
  X^{[{A}} \Sigma^{{B} {C} {D}]} & = 0 \, ,
\end{align}
for some non-zero $3$-tractor $\Sigma^{ABC}$ which is parallel along $\gamma$ for the tractor connection.  Let
$\mbf{u}^a$ denote the weighted velocity associated to $\gamma$.
Picking a scale $\si\in\Gamma(\ce_+[1])$ to compute, and
differentiating $\si^{-1}$ times \nn{eq-inc-conf-circ} (with $\mbf{u}^b\nabla_b$) along
$\gamma$ we obtain that
$$
  U^{[{A}} \Sigma^{{B} {C} {D}]} = 0.
$$
Now differentiating this last display, again with  $\mbf{u}^b\nabla_b$,  we conclude that
$$A^{[{A}} \Sigma^{{B} {C} {D}]} = 0$$ where $A^{A}$ is the
acceleration tractor of $\gamma$. Since $X^{A}$, $U^{A}$ and $A^{A}$
are linearly independent, $\Sigma^{{A}{B}{C}}$ must be
$f\si^{-1}X^{[{A}} U^{B} A^{{C}]}$ for some function $f$ along
$\gamma$. Using the tractor metric to contract this with itself, and comparing with the last statement of Lemma \ref{llem} we
conclude that $f^2$ is constant along $\gamma$, and with no loss of generality, $\Sigma$ can be taken to be given by \eqref{eq-SigmaXUA}. The result now follows from Proposition
\ref{prop-conf-circ-par}.  \hfill $\Box$

\vspace{1cm}

\subsection{A comment on parametrised conformal circles} \label{param}
As is already pointed out in \cite{Bailey1990a}, the parametrised conformal circle equation
\eqref{par-conf-circ} is equivalent to its unparametrised counterpart \eqref{eq-conf-circ} together with the additional condition
\begin{align}\label{eq-conf-circ-par}
  \left( u^b \nabla_b a^a \right) u_a & = 3 \frac{(u \cdot a)^2}{u \cdot u} - \frac{3}{2} ( a \cdot a) - ( u \cdot u ) u^a u^b \Rho_{ab} \, .
\end{align}
In particular, this equation must govern the choice the choice of parametrisation of a conformal circle. In fact,
it is shown in \cite{BEG} that {\em any} curve $\gamma$ on $(M,\bc)$, not necessarily a conformal circle, is projectively parametrised if and only if its acceleration tractor satisfies
\begin{align}\label{eq-A2=0}
  A \cdot A & = 0 \, .
\end{align}
From \eqref{eq-A.A}, we can immediately conclude
\begin{lem}
  Let $\gamma$ be an oriented nowhere-null curve on $(M,\bc)$ with weighted velocity $\mbf{u}^a$ and acceleration $\mbf{a}^a$. Then a density $\sigma \in \Gamma( \mc{E}[1] )$ determines a projective parametrisation of $\gamma$ if and only if, for any $g \in \bc$ with Levi-Civita connection $\nabla$, $\sigma$ satisfies
  \begin{align*}
    \mbf{u}^a \mbf{u}^b \nabla_a \nabla_b \sigma + \mbf{a}^a \nabla_a \sigma - \frac{1}{2} \sigma^{-1} (\mbf{u}^a \nabla_a \sigma)^2+ \mbf{u}^a \mbf{u}^b \Rho_{ab} \sigma \pm \frac{1}{2} \mbf{a}^b \mbf{a}_b \sigma & = 0 \, ,
  \end{align*}
  or equivalently, as a prolonged system,
  \begin{align*}
      & \mbf{u}^a \nabla_a \sigma - \tau = 0 \, ,                                                                                                           \\
      & \mbf{u}^a \nabla_a \tau - \frac{1}{2} \tau^2 \sigma^{-1} + \mbf{u}^a \mbf{u}^b \Rho_{ab} \sigma \pm \frac{1}{2} \mbf{a}^b \mbf{a}_b \sigma = 0 \, ,
  \end{align*}
  whenever $\mbf{u}^a \mbf{u}_a = \pm 1$.
\end{lem}

\subsection{The model and incidence}\label{circ-model}

In the context of conformal circles, we recall from section \ref{nmodel} that the flat model consists of a conformal quadric $\mc{Q}$ embedded in $S^{n+1} = \mbb{P}_+ (\R^{n+2})$ as the image under $\mbb{P}_+$ of the null cone $\mbb{N} \subset \R^{n+2}$ defined by a non-degenerate symmetric bilinear form $h$ of signature $(p+1,q+1)$ on $\R^{n+2}$.

Each orbit of the Lie group $G=\SO(p+1,q+1)$ acting on $\R^{n+2}$ is characterised by the signature of $h$ restricted to the lines through the origin. These may be null, spacelike or timelike. In particular,  we can identify $\mc{Q}$ as the closed orbit of the Lie group $G=\SO(p+1,q+1)$ on $\R^{n+2}$ arising from $\mbb{N}$.

We can similarly describe the space of oriented conformal circles in $\mc{Q}$ as an open orbit of the Lie group $G$ on the space $\mbb{P}_+ \wt{\T}$ of oriented $2$-planes in $\mbb{P}_+ (\R^{n+2})$ where $\wt{\T} := \left\{ \Sigma^{A B C}  \in  \Lambda^3 \R^{n+2}  : \Sigma^{[A B C} \Sigma^{D] E F} = 0 \right\}$ consists of all simple elements of $\Lambda^3 \R^{n+2}$. In indefinite signature, the space of oriented conformal circles space splits into two disjoint connected components according to whether they are spacelike or timelike.

From our previous discussion, especially in connection to Lemma \ref{llem}, the space of oriented spacelike, respectively timelike, conformal circles is the $(3n-3)$-dimensional ray projectivisation $\mbb{P}_+ \T^{+, +, -}$, respectively $\mbb{P}_+ \T^{+,-,-}$, of the smooth variety
\begin{equation}
  \T^{+,\pm,-} := \left\{ \Sigma^{A B C} \in \Lambda^3 \R^{n+2} : \Sigma^{[A B C} \Sigma^{D] E F} = 0 \, , \mathrm{sign}(\Sigma) = (+,\pm,-) \right\} \subset  \wt{\T} \, ,
\end{equation}
where $\mathrm{sign}(\Sigma)$ denotes the signature of an element $\Sigma$ of $\wt{\T}$, that is, the signature of the restriction of $h$ to the span of the factors of $\Sigma$.

A conformal circle here is a circle in $\mc{Q}$ arising from the intersection of $\mc{Q}$ and the oriented $2$-plane defined by an element of $\mbb{P}_+ \T^{+, \pm, -}$.

Finally, the correspondence space between $\mc{Q}$ and $\mbb{P}_+ \T^{+, \pm, -}$ is the $(3n-2)$-dimensional submanifold $\mc{F}^{+, \pm, -}$ of $\mc{Q} \times  \mbb{P}_+ \T^{+, \pm, -}$ defined by the incidence relation $X \wedge \Sigma = 0$, for $[X] \in \mc{Q}$, $[\Sigma] \in \mbb{P}_+ \T^{+, \pm, -}$.
It gives rise to the double fibration
\begin{align*}
  \xymatrix { & \mc{F}^{+, \pm, -}  \ar[dl] \ar[dr] &                            \\
  \mc{Q}      &                                     & \mbb{P}_+ \T^{+, \pm, -} }
\end{align*}
As a bundle over $\mc{Q}$, a fibre of $\mc{F}^{+, \pm, -}$ over a point $x$ in $\mc{Q}$ can be identified as the space of all velocity and acceleration directions of all curves through $x$. The velocity and acceleration vectors up to scale of a curve $\gamma$ in $\mc{Q}$ determine a lift of $\gamma$ to $\mc{F}^{+, \pm, -}$, which descends to a point of $\mbb{P}_+ \T^{+, \pm, -}$ if and only if $\gamma$ is an oriented conformal circle.

Unlike in the projective and the null conformal cases, the metric on $\Lambda^3 \R^{n+2}$ allows us to single out a unique representative $\Sigma \in \T^{+, \pm, -}$ of an element of $\mbb{P}_+ \T^{+, \pm, -}$  by choosing the normalisation $|\Sigma|^2 = \mp 1$.

\subsubsection{Initial conditions for conformal circles}\label{sect-IC-conf-circles}
The discussion above allows us to treat the problem of initial conditions for oriented conformal circles on a conformal manifold in a way similar to the projective and null conformal treatments of Sections \ref{sect-IC-geod} and \ref{sect-IC-n-geod}. The only difference here is that given a point $\mr{x}$, we need to specify a velocity $\mr{u}$ and an acceleration $\mr{a}$ at $\mr{x}$ consistent with initial conditions for the third-order ODE \eqref{eq-conf-circ} governing conformal circles.

\begin{thm}
  Let $(M,\bc)$ be an $n$-dimensional conformal manifold. Fix a point $\mr{x}$ in $M$ so that $X_{\mr{x}}$ is the canonical tractor based at $\mr{x}$. Then, for every non-zero element $\mr{\Sigma}$ of $\Lambda^3 \mc{T}_{\mr{x}}$ that satisfies the incidence relation:
  \begin{align}\label{eq-cc-IC}
    X_{\mr{x}}\wedge \mr{\Sigma} =0 \, ,
  \end{align}
  together with the following conditions
  \begin{enumerate}
    \item $\mr{\Sigma}$ is simple, \label{cond-simple}
    \item $\mr{\Sigma}$ has signature $(+,+,-)$ or $(+,-,-)$ \label{cond-sign},
  \end{enumerate}
  there locally exists a unique unparametrised oriented conformal circle $\gamma$ through $\mr{x}$, which is spacelike if $\mathrm{sign}(\mr{\Sigma})=(+,+,-)$ or timelike if $\mathrm{sign}(\mr{\Sigma})=(+,-,-)$. Further, if $\mr{\Sigma}$ is normalised to $|\mr{\Sigma}|^2 = \mp 1$, then the tractor $\Sigma$ associated to $\gamma$ (via Theorem \ref{main-c}) satisfies $\Sigma_{\mr{x}}= \mr{\Sigma}$.
\end{thm}

\begin{proof}
  The reasoning follows the proofs of Theorems \ref{thm-IC-geod} and \ref{thm-IC-n-geod}. The incidence relation \eqref{eq-cc-IC} together with the condition \ref{cond-simple} yields two vectors $\mr{u}$ and $\mr{a}$ at $\mr{x}$. Condition \ref{cond-sign} tells us that $\mr{u}$ is either spacelike or timelike. The existence and uniqueness of the conformal circle through $\mr{x}$ now follows from the theory of ODE applied to \eqref{eq-conf-circ}.
\end{proof}

\section{Symmetry and first BGG type equations}\label{BGG-sect}

On a Riemannian or pseudo-Riemannian manifold $(M,g)$ a vector field
$k$ is an infinitesimal isometry, or {\em Killing
    vector}, if it satisfies the Killing equation $\cL_k g=0$, which may
be written as $\nabla_{(a}k_{b)}=0$ in terms of the Levi-Civita
connection $\nabla$. More generally a tensor $k\in \Gamma(\odot^kT^*M) $
is called a {\em Killing tensor} if it satisfies the equation
\begin{equation}\label{Kt}
  \nabla_{(a_0} k_{a_1\cdots a_k)}=0.
\end{equation}
A rank-$k$ {\em Killing-Yano tensor} (or form) is a $k$-form $F \in\Gamma(\Lambda^{k} T^*M)$
that satisfies
\begin{equation}\label{KY}
  \nabla F \in \Gamma(\Lambda^{k+1} T^*M).
\end{equation}
Both Killing tensors and Killing-Yano tensors have been used for the
construction of first integrals of geodesics
\cite{Andersson-Blue,Frolov2017}.

There are conformal variants of these equations: the {\em conformal Killing equation}
\begin{equation}\label{cKt}
  \nabla_{(a_0}K_{a_1\cdots a_k)_0}=0,
\end{equation}
on trace-free  rank-$k$ tensors $K$ ($k\geq 1$); and the {\em conformal Killing-Yano equation}
\begin{equation}\label{cKY}
  \operatorname{tf}(\nabla F) \in \Gamma(\Lambda^{k+1} T^*M),
\end{equation}
on $k$-forms $F$, where $\operatorname{tf}$ means the (metric)
trace-free part of the given tensor.
These are known for providing first integrals of null geodesics.

For the cases of $k\geq 2$, the solutions of these equations are
sometimes called {\em hidden symmetries} because they are not
symmetries, in the usual sense, of $(M,g)$ (or `configuration space'),
but rather symmetries of the standard metric Hamiltonian on the
cotangent bundle (or `phase space') which do not descend to
isometries, but nevertheless provide such conserved quantities.  It
turns out that the equations \nn{Kt} and \nn{KY} are projectively
invariant, meaning that they descend to well defined equations on
projective manifolds, provided $k$ and $F$ are assigned a suitable
projective weight.  In fact they are in the class of {\em projective
    first BGG equations}. Similarly if $K$ and $F$ are assigned suitable
conformal weights then \nn{cKt} and \nn{cKY} are conformally invariant
and are, in particular, {\em conformal first BGG equations}.

The Killing type equations above are very specific examples of first BGG equations ;there are a vast number of equations in this class.
Using the characterisation of distinguished curves from
Theorems \ref{main-p}, \ref{main-nullc}, and \ref{main-c} we will show that a suitable
solution to any one of these can lead to first integrals of geodesics
(in the case of projective BGGs) and of conformal circles (in the case of conformal BGGs). Moreover we will show that this fits into a uniform and elegant general theory.

\subsection{Elements of BGG theory}\label{BGGsec}

Recall our model for oriented projective geometry is
$S^{n}=\mathbb{P}_+(\mathbb{R}^{n+1})$. This is a homogeneous space
for $G=\SL(\mathbb{R}^{n+1})\cong \SL(n+1)$, and
$\mathbb{P}_+(\mathbb{R}^{n+1}) =G/P$ where $P$ may be taken to be the
parabolic subgroup stabilising a fixed nominated ray from the origin.

Similarly recall our model oriented conformal geometries of signature
$(p,q)$ is the ray projectivisation of $\mathbb{N}$, where $\mathbb{N}$ is the null quadric
$\mathbb{N}:=\{X\in\mathbb{R}^{n+2} \mid h(X,X) =0\}$
in $\mathbb{R}^{n+2}$ equipped with a fixed symmetric non-degenerate
bilinear form $h$, of signature $(p+1,q+1)$. This $S^p \times S^q$
is acted on transitively by $G=\SO(p+1,q+1)$ and the stabiliser of a
point is again a parabolic subgroup that we will also denote $P$.

Conformal and projective geometries are examples of parabolic
geometries, as studied generally in \cite{CS-book}, and we refer the
reader to that source for general background. Each such geometry is by
definition modelled on a homogeneous manifold $G/P$ where $G$ is a
semi-simple Lie group and $P$ is a parabolic subgroup. It consists of a
manifold $M$ of dimension $\operatorname{dim}(G/P)$ and a
$P$-principal bundle $\cG\to M$ equipped with a canonical {\em Cartan
    connection } $\om$ which is a suitably equivariant
$\mathfrak{g}$-valued 1-form that provides a total
parallelization of $T\cG$. Here $\mathfrak{g}$ denotes the Lie algebra of $G$.
For any representation $\mathbb{U}$ of $P$, one has a corresponding associated bundle $\cG \times_P \mathbb{U}$.
The tractor bundles are the associated
bundles $\cW:=\cG\times_P \mathbb{W}$ where $\mathbb{W}$ is a linear
representation space of $G$ (and hence also of $P$ by restriction)\cite{CG-tams}. On
these the Cartan connection induces a linear connection $\nabla^{\mc{W}}$. In fact this is easily understood.
As for any associated bundle, a section $T$ of $\cW$ is equivalent to a function
\begin{equation}\label{sections}
  t:\cG\to \mathbb{W} \quad \mbox{satisfying} \quad t(u\cdot p)= \rho(p^{-1})t(u)
\end{equation}
for all $u\in \cG$ and $p\in P$, with $\rho$ denoting the
representation of $G$ (restricted here to $P$) on $\mathbb{W}
$. Given a smooth vector field $\xi$ on $M$ we can always find a lift
$\overline{\xi}$ of $\xi$ to $\cG$ that is invariant under the
principal right action of $P$.  Then
\begin{equation}\label{lift-conn}
  ( \overline{\xi} \cdot t +\rho' (\om (\overline{\xi})) t ):\cG\to \mathbb{W}
\end{equation}
is a $P$-equivariant  function, where the representation $\rho':\mathfrak{g}\to \End (\mathbb{W})$ is the derivative of the representation $\rho$.
This is independent of the choice of lift $\overline{\xi}$ and is precisely the equivariant function corresponding to $\nabla^\cW_\xi T$ \cite{CG-tams}. In particular, $T$ is parallel along $\xi$, i.e.\ $\nabla^\cW_\xi T = 0$, if and only if its corresponding $P$-equivariant function $t$ satisfies
\begin{align}\label{par-lift-conn}
  \overline{\xi} \cdot t +\rho' (\om (\overline{\xi})) t & = 0 \, .
\end{align}

For a classical Lie group $G$ the defining (or sometimes called standard)
representation gives the tractor bundle what we term the standard
tractor bundle. Conversely given such a tractor bundle $\cT$ and its
linear tractor connection we can recover the Cartan bundle, as an
adapted frame bundle for $\cT$, and the Cartan connection from
the tractor connection $\nabla^\cT$ on $\cT$, see \cite{CG-tams}, essentially by extracting it from the formula \eqref{lift-conn}.
In particular for projective and conformal geometries we obtain Cartan
connections modelled on the appropriate $(G,P)$ as above, and the
information of the Cartan connection is contained in the respective
tractor connections as introduced above. It is convenient here to use
mainly the same notation for the projective and conformal cases as the
general discussion applies to both.

From \nn{euler} there is a projectively invariant injective bundle map
$$
  \mathbb{X}:T^* M\to \End (\cT)
$$
given by $u_b\mapsto X^A Z_{B}{}^b u_b$. Similarly for conformal geometry there the similar bundle embedding \eqref{injpc}
$$
  \bX: T^*M \to \Lambda^2\cT \subset \End (\cT),
$$
where the tractor metric is used in the obvious way to identify elements of $ \Lambda^2\cT$ with
skew elements of $\End (\cT)$.
In either case sections of $\End (\cT)$ act on tractor bundles in
the obvious tensorial way and so, via each respective $\mathbb{X}$, we have a
canonical action of $T^*M$ on any tractor bundle $\cV$ and this
induces a sequence of invariant bundle maps
\begin{equation}\label{Kostant}
  \partial^* : \Lambda^k T^*M\otimes \cV \to  \Lambda^{k-1} T^*M\otimes \cV,\quad k=1,\cdots,n+1.
\end{equation}
This is the (bundle version of the) Kostant codifferential for
projective, respectively conformal, geometry and satisfies
$\partial^*\circ \partial^*=0$; so it determines subquotient bundles
$\cH_k(M,\cV):=\operatorname{ker}(\partial^*)/\operatorname{im}(\partial^*)
$ of the $\cV$-valued tractor bundles $\Lambda^k T^*M\otimes \cV
$.

Next, for each tractor bundle $\cV=\cG\times_P \mathbb{V}$, with $\mathbb{V}$ irreducible for $G$,
one obtains a so-called \textit{BGG-sequence} \cite{CSS} see also \cite{CD}.
$$
  \cH_0 \stackrel{\cD^{\cV}_0}{\rightarrow} \cH_1 \overset{\cD^{\cV}_1}{\rightarrow}\cdots
  \overset{\cD^{\cV}_{n-1}}{\rightarrow} \cH_n \, .
$$
Here  $\cH_k= \cH_k(M,\cV)$ and
each $\cD^{\cV}_i$ is a linear projectively, respectively conformally, invariant differential operator.

Here we will only be interested in the operator $\cD^\cV=\cD^\cV_0$,
which defines an overdetermined system and is closely related to the
tractor connection $\na$ on $\cV$.  The parabolic subgroup $P\subset
  G$ determines a filtration on $\mathbb{V}$ by $P$--invariant
subspaces. Denoting the largest non--trivial filtration component by
$\V^0\subset \V$, then $\cH_0$ is the quotient $\cV/\cV^0$. Here, $\cV^0$ is the corresponding associated bundle for $\V^0$, and we write $\Pi: \cV \to \cH_0$ for the natural projection.

We recall here the construction of the first BGG operators
$\cD^\cV$, and also the definition of the special class of so
called normal solutions (cf.\ \cite{Leitner}) for these
operators. For the current article we only need the following very
general theorem in the setting of projective or conformal
geometry. If the former, we take $G=\SL(n+1,\mathbb{R})$, $\cG$ the
projective Cartan bundle and invariance means projective
invariance. In the case of conformal geometry we take $G$ to mean
$\SO(p+1,q+1)$, $\cG$ the conformal Cartan bundle and invariance
means conformal invariance.

\begin{thm}[\cite{CGH}] \label{normp}
  Let $\mathbb{V}$ be an irreducible $G$-representation and let $\cV: =\cG\times_P \mathbb{V}$.
  There is a unique  invariant differential operator
  $L: \cH_0 \to  \cV$ such that $\Pi\circ L$ is the identity map on $\cH_0$ and
  $\nabla \circ L$ lies in $\operatorname{ker}(\partial^*)\subset T^*M\otimes \cV$.
  For $\si\in \Gamma(\cH_0)$, $\cD^\cV \si$ is given by projecting $\nabla (L(\si))$ to $\Gamma (\cH_1)$, i.e. $\cD^\cV \si = \Pi(\nabla(L(\si)))$.

  Furthermore  the bundle map
  $\Pi$ induces an injection from the space of parallel sections of
  $\mathcal{V}$ to a subspace of $\Gamma(\mathcal{H}_0)$ which is
  contained in the kernel of the first BGG operator
  \begin{equation}\label{fBGG}
    \cD^\cV : \cH_0 \to \cH_1 \, .
  \end{equation}
\end{thm}

\begin{defn}
  In the setting of the above theorem, elements of this subspace of $\Gamma(\cH_0)$ are called \underline{normal} solutions to the equation $\cD^\cV \si = 0$.
\end{defn}

\noindent The differential operator $L: \cH_0 \to \cV$, in the Theorem, is called a {\em BGG splitting operator}. We
sometimes denote this $L^\cV$ to emphasise the particular tractor
bundle involved.

By definition normal solutions to \eqref{fBGG} are in 1-1
correspondence with parallel sections of the corresponding tractor bundle
$\cV$. On geometries which are flat, according to the tractor/Cartan
connection, all solutions are normal and locally there is
$\operatorname{dim}(\mathbb{V})$-parameter family of such normal solutions.  A projective
manifold is flat in this way if the projective tractor curvature
\eqref{ptract-curv} vanishes and similarly a conformal manifold is flat if the tractor curvature
\eqref{ctract-curv} vanishes.

On curved manifolds only for a very few representations $\mathbb{V}$ of $G$ is
it the case that a solution $\si$ of \eqref{fBGG} is always normal. For example
this happens for the defining representation in both the conformal and
projective cases, and also the dual of that in the latter setting.  In
general $\nabla L(\si)$ is given by curvature terms acting on $L(\si)$
(and this can be reorganised to give an invariant prolongation connection
on $\cV$ so that solutions are in 1-1 correspondence with parallel
sections of $\cV$ \cite{HSSS}). Normal solutions (for which these
curvature terms necessarily annihilate $L(\si)$) often correspond to
interesting geometric conditions on the underlying manifold.  For
example on a projective manifold a parallel maximal rank section of
$\odot^2\cT$ (or the dual bundle) means that there is in the projective class $\bp$ an Einstein
metric with non-zero scalar curvature \cite{ArmstrongI,ArmstrongII,CGH,GH,G-Sil-ckforms}.

\section{Conserved quantities} \label{cq}

Here we give the main theorem concerning the first integrals that arise from the
normal solutions of first BGG equations. For all of the cases there
is a single principle for proliferating these as described in
the  theorem below.

Given a representation space
$\mathbb{V}$ of a  Lie group $G$ let us write $\bigotimes\mathbb{V}$ for the tensor
algebra generated by $\mathbb{V}$, and $\bigodot^m \mathbb{V}$ for the
$m$-fold symmetric tensor product of $\mathbb{V}$. In each case we
take this equipped with the representation of $G$ induced from that on
$\mathbb{V}$.

In the following theorem the meaning of the Lie group $G$, the corresponding Cartan bundle  $\cG$, and the irreducible $G$-representation space will depend
on the setting.
Either: \\
\noindent \hypertarget{p}{(p)} we work on an arbitrary projective manifold $(M^n,\bp)$ and view $\mathbb{R}^{n+1}$ as the defining representation for  $G=\SL(\mathbb{R}^{n+1})\cong \SL(n+1,\mathbb{R})$, and
$$
  \mathbb{W}_0:= \Lambda^2\mathbb{R}^{n+1} ;
$$
or\\
\noindent \hypertarget{n}{(n)} we work on an arbitrary conformal manifold $(M^n,\bc)$ of strictly indefinite signature $(p,q)$,  view $\mathbb{R}^{n+2}$ as the defining representation for $G:=\SO(h)\cong \SO(p+1,q+1)$, where $h$ is a fixed non-degenerate symmetric
bilinear form (on $\mathbb{R}^{n+2}$) of signature $(p+1,q+1)$,
and define
$$
  \mathbb{W}_0:= \Lambda^2\mathbb{R}^{n+2};
$$
or \\
\noindent \hypertarget{c}{(c)} we work on an arbitrary conformal manifold $(M^n,\bc)$, view $\mathbb{R}^{n+2}$ as the defining representation for $G:=\SO(h)\cong \SO(p+1,q+1)$, where $h$ is a fixed non-degenerate symmetric
bilinear form (on $\mathbb{R}^{n+2}$) of signature $(p+1,q+1)$,
and define
$$
  \mathbb{W}_0:= \Lambda^3\mathbb{R}^{n+2}.
$$
Then we have:
\begin{thm}\label{fi-thm}
  Let $\mathbb{V}_1,\cdots ,\mathbb{V}_k$ be irreducible
  representation spaces of $G$,
  $\cV_i=\cG\times_P \mathbb{V}_i$, and $\cD^{\cV_i}$, $i\in
    \{1,\cdots ,k\}$ the corresponding respective first BGG operators.

  For each $i\in \{1,\cdots ,k\}$, suppose that $\si_i$ is a normal
  solution to the first BGG equation
  \begin{equation}\label{BGGi}
    \cD^{\cV_i} \si_i =0 ,
  \end{equation}
  and $m_i\in\mathbb{Z}_{\geq 0}$.
  Then for each copy of the trivial $G$-representation $\mathbb{R}$ in
  \begin{equation}\label{prod}
    (\odot^{m_0} \mathbb{W}_0 )  \otimes (\odot^{m_1} \mathbb{V}_1)\otimes\cdots \otimes (\odot^{m_k}\mathbb{V}_k)
  \end{equation}
  there is a corresponding distinguished curve first integral.
\end{thm}
\noindent Note {\em distinguished curve} here means: unparametrised
geodesic in the setting \hyperlink{p}{(p)}; or unparametrised null geodesic in the
setting \hyperlink{n}{(n)}; or unparametrised conformal circle in the setting \hyperlink{c}{(c)}. In
the following proof of the Theorem $P\subset G$ is a parabolic subgroup
in each case as defined at the beginning of Section \ref{BGGsec}.

\smallskip

\noindent {\em Proof of Theorem \ref{fi-thm}.}
Let $\mathbb{V}_1,\cdots ,\mathbb{V}_k$ and $m_0,\cdots, m_k$ be as in the statement of the Theorem.
To each copy of the trivial $G$-representation $\mathbb{R}$ in
\eqref{prod} there is, in particular, a $G$-epimorphism
\begin{equation}\label{epi}
  \phi: (\odot^{m_0} \mathbb{W}_0 )  \otimes (\odot^{m_1} \mathbb{V}_1)\otimes\cdots \otimes (\odot^{m_k}\mathbb{V}_k)
  \to \mathbb{R},
\end{equation}
where $G$ acts trivially on $\mathbb{R}$.
Let us fix such a map. Each normal solution $\si_i$ ($i=1,\cdots ,k$), to the equations \eqref{BGGi} (on the manifold $(M,\bp)$ in the projective setting \hyperlink{p}{(p)}, or on $(M,\bc)$ for either of the conformal settings \hyperlink{c}{(c)} or \hyperlink{n}{(n)})
is equivalent to a parallel tractor field
$
  L(\si_i)=S_i \in \Gamma(\cG\times_P \mathbb{V}_i)
$, where $L$ is the BGG splitting operator introduced in Theorem \ref{normp}.
Moreover, as discussed above, each of these is equivalent to a $P$-equivariant function
$$
  s_i:\cG\to \mathbb{V}_i, \qquad i \in\{1,\cdots, k\}
$$
that satisfies equation \eqref{par-lift-conn} with $t=s_i$, $\rho'=\rho_i$  the representation of $\g$ on $\V_i$ and $\overline{\xi}$ is the lift of any vector field $\xi$ on $M$ to $\mc{G}$.

On the other hand, according to Theorem \ref{main-p}, Theorem
\ref{main-nullc}, and Theorem \ref{main-c}, along any distinguished
curve $\gamma$ there is a (characterising) parallel tractor $\Sigma\in
  \Gamma(\cG\times_P\mathbb{W}_0)$. On the restriction of $\cG$ that
lies over the trace of $\gamma$ we have that $\Sigma$ is equivalent to
a $P$-equivariant function $s_0:\cG\to \mathbb{W}_0$ that satisfies
equation \eqref{par-lift-conn} with $t=s_0$, $\rho'=\rho_0$ the
representation of $\g$ on $\mathbb{W}_0$, and $\overline{\xi}$ is the
lift of a vector field $\xi$ everywhere tangent to $\gamma$.

Now we form the function
$$
  F=(\odot^{m_0}s_0) \otimes (\odot^{m_1}s_1)\otimes\cdots \otimes (\odot^{m_1}s_1):\cG\to (\odot^{m_0} \mathbb{W}_0 )  \otimes (\odot^{m_1} \mathbb{V}_1)\otimes\cdots \otimes (\odot^{m_k}\mathbb{V}_k).
$$
This is clearly $P$-equivariant, and, by the Leibniz property of the
tractor connection, satisfies \eqref{par-lift-conn} with $t=F$, $\rho'$  the representation of $\g$ on $\mathbb{W}=\mathbb{R}$, and $\overline{\xi}$ is the lift of a vector field $\xi$  everywhere tangent  to $\gamma$. The composition $\phi\circ F$ is then by construction $P$-invariant, and so descends to a function on the trace of $\gamma$. Moreover it is constant along $\gamma$ as for any lift $\overline{\xi}$ of a vector field $\xi$ tangent to $\gamma$, we have
\begin{equation} \label{conserved-cartan}
  \overline{\xi}\cdot \phi(F)= \phi(\overline{\xi}\cdot F)= \phi(\overline{\xi}\cdot F + \rho(\om (\overline{\xi}))F )=0 ,
\end{equation}
where we have used that $\phi$ is simply a fixed linear homomorphism on
the values of $F$ that intertwines the product \eqref{prod} with the trivial representation. \hfill $\Box$

There is  an equivalent  way to prove the Theorem which introduces another object, but
which is useful for applying the Theorem. Since the Cartan connection is $\mathfrak{g}$-valued
it follows easily that any $G$-epimorphism $\phi$, as in \eqref{epi}, determines a
corresponding parallel tractor $T$ field taking values in
$$
  (\otimes^{m_0} \cW_0^*)\otimes  ( \otimes^{m_1} \cV_1^*) \otimes\cdots \otimes( \otimes^{m_k} \cV_k^*)
$$
where $\cW^*_0:=\cG\times_P \mathbb{W}^*_0$, and otherwise we continue the notation above.
The quantity
\begin{equation}\label{Teq}
  T(\odot^{m_0}\Sigma,\odot^{m_1}S_1,\cdots ,\odot^{m_k}S_k)
\end{equation}
is thus constant along any unparametrised geodesic $\gamma$, and this is the first integral.
\begin{rem}
  There is no claim that different $G$-homomorphisms \eqref{epi}
  necessarily yield functionally independent first integrals. Indeed
  for any case where $m_0=0$ the proof goes through without assuming
  $\gamma$ is a distinguished curve, and we thus conclude that any
  $G$-homomorphism \eqref{epi} determines a first integral for \underline{all curves}; these are all functionally equivalent and trivial
  as first integrals. It is easy to understand. In each such case the
  quantity $T(\odot^{m_1}S_1,\cdots ,\odot^{m_k}S_k)$ is constant on
  $M$, and thus, if non-zero, may without loss of generality be taken to be the constant function 1.
\end{rem}

  \subsection{A general procedure for proliferating examples for conformal circles} \label{gen_proc}
  
We illustrate here that finding bundles $\cV_i$, as in Theorem
\ref{fi-thm}, that can yield non-trivial conserved quantities is not
difficult. We show that for each $G$-irreducible part of $S^{m_0}
\mathbb{W}_0$, that corresponds to a non-trvial part of $S^{m_0}
\Sigma$, one can associate such a bundle (and hence BGG equation and
solution). We should emphasise that this is just one choice of bundle,
$\cV_i$ and typically there are many more possibilites.

In what follows, we use the term \emph{non-trivial irreducible part} to
mean a tensor part of $S^{m_0} \Sigma$ that is in general not zero and which arises from a $G$-irreducible part of $S^{m_0}
\mathbb{W}_0$ that is not 1-dimensional (i.e. is not a ``complete contraction'').

\begin{thm}\label{genid}
  Fix $M$ to be the model for conformal geometry, i.e. the sphere
  $S^n$ equipped with the conformal structure induced by the round
  metric.  Let $\gamma$ be a conformal geodesic on $M$ characterised
  by the 3-tractor $\Sigma^{ABC}$.  Then for each $m_0 \in \mathbb{N}$
  and each nontrivial non-trivial irreducible part of $S^{m_0} \Sigma$, there is
  a corresponding nontrivial conserved quantity.
\end{thm}

\begin{proof}
  Suppose $\mathcal{B}$ is the bundle associated with a nontrivial
  $G$-irreducible part of $S^{m_0} \mathbb{W}_0$, and
  $S\in\Gamma(\mathcal{B})$ is a corresponding nontrivial irreducible
  part of $S^{m_0} \Sigma$.  Choose some point $x\in\gamma$, and let
  $p\in \pi^{-1}(x)$, where $\pi : \mathcal{B} \to M$.  Recall $S$ is
  given by some $P$-equivariant function $s : \cG \to \mathbb{B}$,
  where $\mathbb{B}$ is a $G$-representation and $\mathcal{B} =
  \mathcal{G} \times_P \mathbb{B}$.  Hence $s(p) \in \mathbb{B}$.  At
  the point $p$, one may find an element $t \in \mathbb{B}^*$ such
  that $\langle s(p), t \rangle \neq 0$, where $\langle \cdot , \cdot
  \rangle$ denotes the dual pairing between the representations
  $\mathbb{B}$ and $\mathbb{B}^*$.  This element $t$ of the dual
  representation may be extended to an equivariant function on the
  fibre, i.e. an element of $\mathcal{B}|_x$, and this function may be
  further extended by parallel transport to a section $T$ of
  $\mathcal{B}$.  Then by (\ref{conserved-cartan}), $S\cdot T$ is
  constant along any conformal circle $\gamma$, and is, by
  construction, nontrivial.
  
\end{proof}

Each such conserved quantity $S\cdot T$ is, by construction, a
polynomial in the velocity and acceleration of the conformal circle.
In the terminology of Section \ref{circ-model}, this is a constraint
on the space $\tilde{\mathbb{T}}$ of simple 3-tractors $\Sigma^{ABC}$.
But on the model, such 3-tractors are in bijective correspondence with
conformal circles, and so this polynomial is exactly a constraint on
the space of conformal circles.

Finally we point out that although we have discussed here the model,
it follows that in general a non-zero parallel section of the bundle
$\mathcal{B}^*$ (with this bundle the as in the proof of Theorem \ref{genid}
above, but now on any conformal manifold $(M,\cc)$) determines a
non-trivial first integral of conformal circles.

\subsection{A digression on notation and Young symmetries}\label{young-sec}

Given a vector space $\mathbb{V}$ and $s,t\in \mathbb{Z}_{\geq 1}$ we will write
$$\mathbb{V}^{(s,s,\cdots ,s)}\subset
  \otimes^t(\odot^s\mathbb{V})
$$ to be the subspace of tensors in $ \otimes^t(\odot^s\mathbb{V})$
that vanish upon symmetrisation over any $s+1$ indices, in the sense
of abstract indices. This subspace is an irreducible component with
respect to the group $\GL(\mathbb{V})$ acting in the standard way on
$\otimes^{st}\mathbb{V} $ (and $\mathbb{V}^{(s,s,\cdots ,s)}$ is the
image of a Young projector on $\otimes^{st}\mathbb{V} $)
\cite{Fulton-Harris,Penrose-Rindler-v1}. This notation is also used in
\cite{Go-Leist} where there is further discussion.  We write
$\mathbb{V}_{(s,s,\cdots ,s)}\subset \otimes^t(\odot^s\mathbb{V}^*) $
for the dual tensor space, with the same symmetries but now
constructed using the vector space $\mathbb{V}^*$ dual to
$\mathbb{V}$.

Similarly given the same  vector space $\mathbb{V}$ we write
$$\mathbb{V}^{[s,s,\cdots ,s]}\subset
  \odot^s(\Lambda^t\mathbb{V})
$$
to be the subspace of tensors in $\odot^s(\Lambda^t\mathbb{V} )$
that vanish upon alternation over any $t+1$ indices.
This subspace also is an irreducible component with
respect to the group $\GL(\mathbb{V})$ acting in the standard way on
$\otimes^{st}\mathbb{V} $. In fact it is well known that there is an $\GL(\mathbb{V})$-isomorphism
$$
  \mathbb{V}^{[s,s,\cdots ,s]} \cong \mathbb{V}^{(s,s,\cdots ,s)}
$$ (and $\mathbb{V}^{[s,s,\cdots ,s]}$ is the image of another Young
projector on $\otimes^{st}\mathbb{V} $ that simply gives a different
realisation of the same representation). We write
$\mathbb{V}_{[s,s,\cdots ,s]}\subset \odot^s(\Lambda^t\mathbb{V}^*) $
for the dual tensor space, again constructed the same way but starting now with $\mathbb{V}^*$.

We will carry these notations onto tractor bundles in the obvious way. So in the setting of either projective or conformal tractors
$$
  \cT^{[s,s,\cdots ,s]},
$$
for example, will mean the subbundle of $\odot^s(\Lambda^t \cT)$ with fibre
$(\cT_x)^{[s,s,\cdots ,s]}$ at any $x\in M$.

\subsection{The first integrals of affine geodesics and projective curves}\label{pfi}

Recall that in this case we view $\mathbb{R}^{n+1}$ as the defining representation
for $G=\SL(\mathbb{R}^{n+1})\cong \SL(n+1,\mathbb{R})$ and
$$
  \mathbb{W}_0:= \Lambda^2\mathbb{R}^{n+1}
$$ as an irreducible $G$-representation space.

Since affine connections determine a projective structure it suffices to study the invariants on any
projective manifold $(M^n,\bp)$ and write
$\cG$ for the projective Cartan geometry modelled on $(G,P)$ as discussed above.

Using Weyl's invariant theory \cite{Weyl-inv} we know that $\phi$ is
determined by the volume form on $\mathbb{R}^{n+1}$, as preserved by
$G=\SL(\mathbb{R}^{n+1})$, and traces. Equivalently in any example the
formula for $T$ is constructed using the tractor volume form, its dual, and
the identity $\delta^A_B$.

Thus in summary and informally the construction of first integrals is as
follows.  Any normal solution of a first BGG equation provides (and is
equivalent to) a parallel tractor field. Given any collection of
parallel tractor fields, including the tractor volume form and its
tensor powers, we form first integrals by simply contracting these
into tensor powers of $\Sigma$. We construct some simple examples as follows.

\subsubsection{The classical first integrals -- Killing tensors}\label{class}

As mentioned above Killing tensors provide first integrals along
geodesics. This is simply because the velocity $u$ of an affinely parametrised
geodesic satisfies $\nabla_uu=0$ and hence for any Killing tensor
$k_{b\cdots c}$ the quantity $ u^b\cdots u^ck_{b\cdots c} $ is
constant along the geodesic -- here, we view $k_{b\cdots c}$ as an unweighted tensor.

This is recovered from the Theorem \ref{fi-thm} as follows. The representation and corresponding tractor bundle for Killing tensors can be read off from standard representation theory as discussed in \cite{BCEG}.
In this case the BGG splitting operator $k\mapsto L(k)$ is a map
$$
  \Gamma(\odot^sT^*M(2s))\ni k_{b_1\cdots b_s} \mapsto  \mathbb{K}_{\alpha_1\cdots \alpha_s\beta_1\cdots \beta_s}\in \Gamma(\cT_{[s,s]})\subset  \Gamma(\otimes^{2s}\cT^*),
$$
where $ \mathbb{K}_{\alpha_1\cdots \alpha_s\beta_1\cdots \beta_s}$
is a (weight zero) tractor that is skew on each pair
$\alpha_i\beta_i$, $i=1,\cdots ,s$.
From the
sequence \eqref{euler} it follows easily that the map $\Pi$ (of
Theorem \ref{normp}) that gives a left inverse to $L$ is obtained (up
to multiplication by a non-zero constant) by contracting
$\mathbb{X}^{\alpha_i\beta_i}_{a_i}$, $i=1,\ldots ,s$, into
$\mathbb{K}_{\alpha_1\cdots \alpha_s\beta_1\cdots \beta_s}$ and, again using \eqref{euler},
it follows that
$$
  K_{\beta_1\cdots \beta_s}: = X^{\alpha_1}\cdots X^{\alpha_s} \mathbb{K}_{\alpha_1\cdots \alpha_s\beta_1\cdots \beta_s}
$$
satisfies
$$
  K_{\beta_1\cdots \beta_s}= \tilde{c}\cdot Z_{\beta_1}{}^{b_1} \cdots Z_{\beta_s}{}^{b_s} k_{b_1\cdots b_s},
$$
for some constant $\tilde{c}\neq 0$.
Thus along an unparametrised geodesic with weighted velocity $\mbf{u}^a$ we have
\begin{equation}\label{key-if}
  \Sigma^{\alpha_1\beta_1}\cdots \Sigma^{\alpha_s\beta_s}\mathbb{K}_{\alpha_1\cdots \alpha_s\beta_1\cdots \beta_s} = c\cdot \mbf{u}^{b_1}\cdots \mbf{u}^{b_s}k_{b_1\cdots b_s},
\end{equation}
for some constant $c\neq 0$.
Now according to the Theorem \ref{fi-thm}, if $k_{b_1\cdots b_s}$ is a
normal solution of the Killing equation \eqref{Kt} then the display
\eqref{key-if} is a first integral of unparametrised geodesics. In
this case we have recovered the well known quantity on the right hand side.

\begin{rem}\label{non-normal} There is an interesting observation here. Evidently
  \eqref{key-if} defines a first integral even if the solution is not
  necessarily normal, because \eqref{key-if} recovers the usual first integral associated to
  Killing tensors.  In the case where $k_{b_1\cdots b_s}$ is a solution
  of the Killing equation \eqref{Kt} but not necessarily normal then
  the tractor $\mathbb{K}$ is no longer parallel along the curve but
  rather $\nabla_u \mathbb{K}$ is given by some algebraic action of
  the tractor curvature and its derivatives on $\mathbb{K}$
  \cite{HSSS}. Evidently the contraction with
  $\Sigma^{\alpha_1\beta_1}\cdots \Sigma^{\alpha_s\beta_s}$
  annihilates these terms.  Using a slightly different splitting
  operator and prolongation procedure, an algorithm for explicitly
  computing these curvature terms was found recently in
  \cite{Go-Leist}. Using this it is easily seen explicitly that the
  given curvature terms are indeed annihilated by the contraction with the
  $X^{\alpha_1}\cdots X^{\alpha_s}$ implicit in the
  $\Sigma^{\alpha_1\beta_1}\cdots \Sigma^{\alpha_s\beta_s}$
  contraction. In light of the examples presented later in this paper,
  it seems likely that a similar argument will show that
  Theorem \ref{fi-thm} will extend to many cases of non-normal BGG
  solutions and also to solutions of other geometric equations that
  have the same leading symbol. This requires an extension of the
  programme initiated in \cite{Go-Leist} or a theory that establishes
  similar results.
\end{rem}

\subsubsection{The general case -- Killing tensors from BGG solutions}\label{truelyhidden}

The first integrals for geodesics found using Theorem \ref{BGGi} (with
the assumptions \hyperlink{p}{(p)}) are, by construction, polynomial on the fibres of
$TM$.  On the other hand it is a classical result that any first
integral of geodesics that is polynomial on the fibres of $TM$ is a
sum of a constant function and a finite number of ``classical first
integrals'' as in Section \ref{class} above. This is easily seen
directly for the geodesic first integrals \eqref{Teq} from the
Theorem \ref{BGGi}. We need first a preliminary fact.

\begin{prop}\label{BGGtoK-prop} On a manifold with a projective structure $(M,\bp)$
  let $Q= Q_{\alpha_1\beta_1\cdots \alpha_{m_0}\beta_{m_0}}$ be a
  parallel tractor field taking values in
  $
    \otimes^{m_0} \cW_0^*
  $.
  Then
  \begin{align}\label{captKill}
    k_{a_1\cdots a_{m_0}}:= \mathbb{X}_{(a_1}^{\alpha_1\beta_1}\cdots \mathbb{X}_{a_{m_0})}^{\alpha_{m_0}\beta_{m_0}}Q_{\alpha_1\beta_1\cdots \alpha_{m_0}\beta_{m_0}}
  \end{align}
  is a normal Killing tensor, i.e. a solution to \eqref{Kt} with
  $L(k)$ parallel for the normal tractor connection.
\end{prop}
\begin{proof}
  Observe that for any section $\mbf{u}^a\in \Gamma(TM(-2))$ the tractor field
  \begin{equation}\label{prepol}
    \mbf{u}^{a_1}\cdots \mbf{u}^{a_{m_0}}\mathbb{X}_{a_1}^{\alpha_1\beta_1}\cdots \mathbb{X}_{a_{m_0}}^{\alpha_{m_0}\beta_{m_0}}
  \end{equation}
  takes value in $\odot^{m_0}(\Lambda^2\cT)$, but in any scale $\mbf{u}^a
    \mathbb{X}_{a}^{\alpha\beta}$ is simple: $\mbf{u}^a
    \mathbb{X}_a^{\alpha\beta}=2X^{[\alpha}U^{\beta]}$, where $U^\beta:=\mbf{u}^bW_b^B$ in the
  notation of Section \ref{proj-sect}. It follows at once  that skewing
  \eqref{prepol} over any three indices will annihilate it and so
  $$
    \mbf{u}^{a_1}\cdots \mbf{u}^{a_{m_0}}\mathbb{X}_{a_1}^{\alpha_1\beta_1}\cdots \mathbb{X}_{a_{m_0}}^{\alpha_{m_0}\beta_{m_0}}\in \cT^{[m_0,m_0]}.
  $$ Thus in the contraction in \eqref{captKill} nothing is changed if
  we replace $Q$ with $P_{[m_0,m_0]}(Q)$. Here $P_{[m_0,m_0]}$ is the
  natural projection from $\otimes^{m_0} \cW_0^*$ to
  $\cT_{[m_0,m_0]}$.
  But $P_{[m_0,m_0]}(Q)$ is a parallel section of the irreducible
  tractor bundle $\cT_{[m_0,m_0]}$ and acting on this contraction with $
    \mathbb{X}_{(a_1}^{\alpha_1\beta_1}\ldots
    \mathbb{X}_{a_{m_0})}^{\alpha_{m_0}\beta_{m_0}}$ recovers (up to a
  non-zero constant multiple) the usual BGG projection $\Pi$, as follows
  easily form the composition series \eqref{euler}. Since $L$ is the
  splitting operator $L(k)= P_{(m_0,m_0)}(Q)$, the result follows from
  Theorem \ref{normp}.
\end{proof}

The use of this is as follows. Suppose that on a projective manifold
we have normal first BGG solutions $\si_i$, $i=1,\cdots, k$, and a
homomorphism $\phi$ as in \eqref{epi}. Then we have the corresponding
parallel tractors $S_i$, $i=1,\cdots, k$, and $T$ (as in \eqref{Teq}), and
$$
  Q:= T(\cdot ,\odot^{m_1}S_1,\cdots ,\odot^{m_k}S_k)
$$ is a parallel tractor on $M$ taking values in $\otimes^{m_0}
  \cW_0^*$. Thus from Proposition \ref{BGGtoK-prop} we obtain a
corresponding normal Killing tensor and this is non-trivial if and
only if the first integral \nn{Teq} is non-trivial.

\subsubsection{A sample BGG equation} \label{case-BGG2}

We illustrate the above with a simple case that also reveals a  further result.
On projective densities $\tau\in
  \Gamma(\ce(2))$ the first projective BGG equation is
\begin{equation} \label{3rd}
  \nabla_{(a} \nabla_{b} \nabla_{c)} \tau + 4 \Rho_{(ab} \nabla_{c)} \tau + 2 \tau \nabla_{(a} \Rho_{bc)} = 0.
\end{equation}
This equation and its importance is discussed in some detail in
e.g.\ \cite{CG-proj-Ein,CGH-jlms,MaGo-Mat}.
In this case the BGG splitting operator is a second-order differential operator $L:\ce(2)\to \ce_{(\alpha \beta)}$ given by
\begin{equation}\label{3rd-split}
  \tau\mapsto L(\tau)=\frac{1}{2}D_{\alpha}D_\beta \tau \, ,
\end{equation}
(cf.\ \cite[Section 3.3]{CG-proj-Ein}) where $D_\alpha : \mc{E}(w) \rightarrow \mc{E}(w-1)$ is the projectively invariant Thomas-D differential operator on weighted tractors defined by $D_\alpha \sigma = w Y_\alpha \sigma + Z_\alpha {}^a \nabla_a \sigma$ (with $\nabla $ the coupling of the tractor connection with the affine connection corresponding to the splitting).
For convenience let us write $H_{\alpha\beta}:=\frac{1}{2}D_{\alpha}D_\beta \tau $.

Thus normal solutions to \eqref{3rd} correspond to $H=L(\tau)$
parallel and in this case it follows at once from \eqref{fi-thm} that
\begin{equation}\label{fi-3rd}
  \Sigma^{\alpha_1\beta_1}\Sigma^{\alpha_2\beta_2}H_{\alpha_1\alpha_2}H_{\beta_1\beta_2}
\end{equation}
is a first integral for unparametrised geodesics. It is
straightforward to see this is not trivial in general. In fact
$L(\tau)$ can be definite; this is exactly the case of there being
a Levi-Civita connection in the projective class $\nabla^g\in \bp$
where $g$ is a definite signature Einstein metric
\cite{ArmstrongI,ArmstrongII,CGMacbeth-Ein}.
Thus
$$
  k_{ab}:=\X_{(a}^{\alpha_1\beta_1}\X_{b)}^{\alpha_2\beta_2}H_{\alpha_1\alpha_2}H_{\beta_1\beta_2}
$$
is in general a non-trivial normal Killing tensor.

Now the sup rising aspect is that, as for the case of Killing tensors
(see Remark \ref{non-normal}), a stronger result is
available. Normality is not required, it is sufficient that $\tau$
solve \eqref{3rd}:
\begin{thm}\label{3rd-thm}
  Suppose that \( \tau \in \Gamma(\mathcal{E}(2) )\) solves the third-order equation \eqref{3rd}. Then with $H:=L(\tau)$, as in \eqref{3rd-split}, the quantity
  \begin{align}\label{cons-qty}
    \Sigma^{\alpha_1\beta_1}\Sigma^{\alpha_2\beta_2}H_{\alpha_1\alpha_2}H_{\beta_1\beta_2}
  \end{align}
  is a first integral along unparametrised geodesics, where $\Sigma$ is as in Theorem \ref{main-p}. Moreover
  \begin{align}\label{eq-Killing-tau}
    \Gamma(\odot^2 T^*M(4)) \ni k_{bc}:=\tau \, \nabla_b \nabla_c \tau + 2 \, \Rho_{bc} \tau^2 - \frac{1}{2} \left( \nabla_b \tau \right)\left( \nabla_c \tau \right)
  \end{align}  is a Killing tensor, in that it satisfies the equation \eqref{Kt}.
\end{thm}
\begin{proof}
  We first calculate an explicit formula for \eqref{cons-qty}. Computing \eqref{3rd-split} yields
  \begin{align*}
    H_{\alpha_1 \alpha_2} = \tau Y_{\alpha_1} Y_{\alpha_2}
    + \nabla_{c} \tau Y_{(\alpha_1} Z_{\alpha_2)} {}^{c}
    + Z_{\alpha_1} {}^{a} Z_{\alpha_2} {}^{b} \left( \frac{1}{2} \nabla_a \nabla_b \tau + \Rho_{ab} \tau \right) \, .
  \end{align*}
  Then, we have
  \begin{equation*}
    \begin{split}
      \Sigma & {}^{\alpha_1\beta_1} H_{\alpha_1\alpha_2} \\
      &= \mbf{u}^{a} \left( X^{\alpha_1} W^{\beta_1} {}_{a} - X^{\beta_1} W^{\alpha_1} {}_{a}  \right)
      \left(Y_{\alpha_1} Y_{\alpha_2}  \tau
      + Y_{(\alpha_1} Z_{\alpha_2)} {}^{b} \nabla_{b} \tau
      + Z_{\alpha_1} {}^{b} Z_{\alpha_2} {}^{c} \left( \frac{1}{2} \nabla_b \nabla_c \tau + \Rho_{bc} \tau \right) \right) \\
      &=  \mbf{u}^{a} \tau \, W^{\beta_1} {}_{a} Y_{\alpha_2}
      + \frac{1}{2}  \mbf{u}^{a} \nabla_b \tau \, W^{\beta_1} {}_{a} Z_{\alpha_2} {}^{b}
      - \frac{1}{2}  \mbf{u}^{a} \nabla_a \tau \, X^{\beta_1} Y_{\alpha_2}
      -  \mbf{u}^{a} \left( \frac{1}{2} \nabla_{a} \nabla_{c} \tau
      + \Rho_{ac} \tau \right) X^{\beta_1} Z_{\alpha_2} {}^{c}.
    \end{split}
  \end{equation*}
  Contracting this section of $\End(\cT)$ with itself yields ($-1$ times):
  \begin{align} \label{eta}
    \eta := \Sigma^{\alpha_1 \beta_1} \Sigma^{\alpha_2 \beta_2} H_{\alpha_1 \alpha_2} H_{\beta_1 \beta_2} = \tau \, \mbf{u}^{a}  \mbf{u}^{b} \nabla_a \nabla_b \tau + 2 \, \mbf{u}^a  \mbf{u}^b \Rho_{ab} \tau^2 - \frac{1}{2} \left(  \mbf{u}^a \nabla_a \tau \right)^2.
  \end{align}
  Now differentiating
  \eqref{eta} along \( \gamma \), and using $ \mbf{u}^a\nabla_a  \mbf{u}^b=0$, we have:
  \begin{align*}
    \mbf{u}^{c} \nabla_{c} \eta & = (  \mbf{u}^c \nabla_c \tau )  \mbf{u}^a  \mbf{u}^b \nabla_a \nabla_b \tau + \tau  \mbf{u}^a  \mbf{u}^b  \mbf{u}^c \nabla_c \nabla_a \nabla_b \tau + 2 \, \mbf{u}^a  \mbf{u}^b  \mbf{u}^c (\nabla_c \Rho_{ab}) \tau^2 \\
                                & \qquad \qquad + 2 \, \mbf{u}^a  \mbf{u}^b  \mbf{u}^c \Rho_{ab} (\nabla_c \tau^2 )- ( \mbf{u}^a \nabla_a \tau)  \mbf{u}^a  \mbf{u}^c \nabla_c \nabla_a \tau                                                             \\
                                & = \tau \mbf{u}^a  \mbf{u}^b  \mbf{u}^c \left( \nabla_a \nabla_b \nabla_c \tau + 2 \, \tau \nabla_a \Rho_{bc} + 4 \, \Rho_{ab} \nabla_c \tau \right)                                                                    \\
                                & = 0,
  \end{align*}
  since \( \tau \) was assumed a solution of \eqref{3rd}. This calculation may also be viewed as the verification that $\nabla_{(a}k_{bc)}=0$.
  Indeed, from the definition \eqref{eq-Killing-tau}, we have
  \begin{align*}
    \nabla_{(a} k_{b c)} & =  \tau \left( \nabla_{(a} \nabla_b \nabla_{c)} \tau + 2 \, \tau \nabla_{(a} \Rho_{bc)} + 4 \, \Rho_{(ab} \nabla_{c)} \tau \right)    \, .
  \end{align*}
  Thus, $ \mbf{u}^{c} \nabla_{c} \eta  = \mbf{u}^a  \mbf{u}^b  \mbf{u}^c (\nabla_a k_{bc}) = 0$, and since this is true for any geodesic, we conclude that $\nabla_{(a} k_{b c)} =0$, i.e.\ $k_{ab}$ is a Killing tensor.
\end{proof}

\subsection{The first integrals of null geodesics}\label{nullfi}

On an indefinite signature pseudo-Riemannian manifold, or the conformal
structure $(M,\bc)$ that it determines, the Theorem \ref{fi-thm} uses
solutions of conformal first BGG equations to generate first
integrals along null geodesics. This is the setting \hyperlink{n}{(n)} for that Theorem so we
view $\mathbb{R}^{n+2}$ here as the defining representation for $G:=\SO(h)$, where $h$ is a fixed non-degenerate symmetric
bilinear form on $\mathbb{R}^{n+2}$ of signature $(p+1,q+1)$,
and define
$$
  \mathbb{W}_0:= \Lambda^2\mathbb{R}^{n+2}.
$$
The situation turns out to be closely analogous to that in Sections \ref{class} and \ref{truelyhidden}
above, so we shall be brief.

In this case $\phi$ is constructed from the bilinear form $h$ and the
compatible volume form on $\mathbb{R}^{n+2}$, as preserved by
$G=\SO(h)$, and traces. Equivalently, in any example the formula for $T$
is constructed using the tractor metric and its inverse, the tractor
volume form and the identity $\delta^A_B$.

The classical results surround primarily conformal Killing tensors,
i.e. solutions $k_{b_1\cdots b_s}$ of the first BGG equation
\eqref{cKt}.  As for the cases above the representation and
corresponding tractor bundle for conformal Killing tensors can be read
off from standard representation theory as discussed in
\cite{BCEG,CGH-jlms}.  The conformal splitting operator $k\mapsto
  L(k)$ in this case is a differential operator
$$
  \Gamma(\odot^sT^*M[2s])\ni k_{b_1\cdots b_s} \to  \mathbb{K}_{A_1\cdots A_sB_1\cdots B_s}\in
  \Gamma(\cT_{[s,s]_0})\subset  \Gamma(\otimes^{2s}\cT)
$$ where $\cT$ denotes the conformal standard tractor bundle, $
  \mathbb{K}_{A_1\cdots A_sB_1\cdots B_s}$ is a (weight zero) trace-free
tractor field that is skew on each pair $A_iB_i$, $i=1,\cdots ,s$. We
write $\cT_{[s,s]_0}$ to indicate the subbundle of $\cT_{[s,s] }$
consisting of tractors that are trace free (with respect to the
conformal tractor metric).

It is easily verified that for normal solutions of the conformal Killing
equation the standard first integral $u^{a_1}\cdots
  u^{a_s}k_{a_1\cdots a_s}$ arises from
\begin{equation}\label{class-cfi}
  \Sigma^{A_1B_1}\cdots \Sigma^{A_sB_s}\mathbb{K}_{A_1\cdots A_sB_1\cdots B_s}.
\end{equation}
Thus, in analogy with the observation in Remark \ref{non-normal}, it is
again the case that \eqref{class-cfi} is conserved along null
geodesics even if the solution $k$ to \eqref{cKt} is not normal,
i.e. $\mathbb{K}=L(k)$ is not parallel.

Also there is an analogue of Proposition \ref{BGGtoK-prop}:
\begin{prop}\label{BGGtoK-cprop} On a manifold with an indefinite  conformal structure $(M,\bc)$,
  let $Q= Q_{A_1B_1\cdots A_{m_0}B_{m_0}}$ be a
  parallel tractor field taking values in
  $
    \otimes^{m_0} \cW_0^*
  $.
  Then
  \begin{equation}\label{captcKill}
    k_{a_1\cdots a_{m_0}}:= \mathbb{X}_{(a_1}^{A_1B_1}\cdots \mathbb{X}_{a_{m_0})}^{A_{m_0}B_{m_0}}Q_{A_1B_1\cdots A_{m_0}B_{m_0}}
  \end{equation}
  is a normal conformal Killing tensor, i.e. a solution to \eqref{cKt} with
  $L(k)$ parallel for the normal conformal tractor connection.
\end{prop}
\begin{proof}
  The proof is almost identical to that for Proposition
  \ref{BGGtoK-prop}. The additional ingredient is that in this case $\odot^{m_0}\Sigma$ is trace-free with respect to the tractor metric because $\Sigma$ is totally null, as observed in Section \ref{null-sec}, and so, using also that $\Sigma$ is simple we have
  $\odot^{m_0}\Sigma\in \Gamma(\cT^{[m_0,m_0]_0})$ along any null curve.
\end{proof}

Thus in the setting \hyperlink{n}{(n)}, the first integrals coming from Theorem
\ref{fi-thm} may be viewed as arising from (normal) conformal Killing
tensors, but these conformal Killing tensors are, in general, arising
from other BGG solutions via Proposition \ref{BGGtoK-cprop}.

\subsection{The first integrals of conformal circles}\label{ncfi}

On a pseudo-Riemannian manifold of any signature, or the conformal
structure $(M,\bc)$ that it determines, Theorem \ref{fi-thm} uses
solutions of conformal first BGG equations to generate first integrals
along conformal circles. This is the setting \hyperlink{c}{(c)} for that Theorem so
we view $\mathbb{R}^{n+2}$ as the defining representation for
$G:=\SO(h)$, where $h$ is a fixed non-degenerate symmetric bilinear
form on $\mathbb{R}^{n+2}$ of signature $(p+1,q+1)$, and now define
$$
  \mathbb{W}_0:= \Lambda^3\mathbb{R}^{n+2}.
$$

Again in this case $\phi$ is determined by the bilinear form $h$ and the
compatible volume form on $\mathbb{R}^{n+2}$, as preserved by
$G=\SO(h)$, and traces. Equivalently, in any example the formula for $T$
is constructed using the tractor metric and its inverse, the tractor
volume form, and the identity $\delta^A_B$.

Thus from the point of view of Theorem \ref{fi-thm} and its general
application there is little difference from the setting \hyperlink{n}{(n)}
above. However an important difference arises in that conformal
Killing tensors no longer have a distinguished role as there.

Even na\"{\i}vely some significant difference is to be expected as the
first integrals found by Theorem \ref{fi-thm} will, by construction,
be (pointwise) polynomial in the velocity {\em and the acceleration}
of the given distinguished curve. However we can see this clearly
using the construction directly, as follows. Recall that for a
conformal circle $\gamma$ the characterising tractor $\Sigma$ is a 3-tractor that, according to the normal tractor connection, is
parallel along $\gamma$. Using that $\Sigma$ is simple and arguing in a similar way to the previous cases we have that
$$
  \otimes^{s}\Sigma \in \Gamma ( \cT^{[s,s,s]})
$$
along $\gamma$. But now the difference is $\Sigma$ does not satisfy
any analogue of the nilpotency \eqref{nil-eq} and $\odot^{s}\Sigma$ does not
take values in a $G$-irreducible tractor bundle if $s> 1$.
We obtain the different irreducible components of $\odot^{s}\Sigma$  by splitting it
into its various  trace-free and trace parts.
The
distinct irreducible components of $\odot^{s}\Sigma$ can then pair
with parallel tractors of distinct tensor type, and thus with the
prolongations of solutions to corresponding distinct first BGG
equations.

\subsubsection{A basic example} \label{ob-ex}

For conformal circles the simplest application of Theorem \ref{fi-thm}
is on a conformal manifold equipped with a tractor $3$-form $\mathbb{K}_{ABC}\in \Gamma(\Lambda^3\cT^*)$
that is parallel for the normal conformal tractor connection. Then clearly
\begin{equation}\label{Sig-ex}
  \Sigma^{ABC}\mathbb{K}_{ABC}
\end{equation}
is necessarily constant along any conformal circle.

The composition series for $\Lambda^3\cT^*$ is
$$
  \Lambda^3\cT^*= \ce_{[bc]}[3]\lpl \left( \ce_{[abc]}[3] \oplus \ce_a[1] \right) \lpl \ce_{[bc]}[1] .
$$
Thus the first integral arises from a solution to the first BGG equation on the projecting
part $\ce_{[bc]}[3] $. In terms of a metric for the conformal class,
this BGG equation is the conformal Killing-Yano equation (or conformal
Killing form equation) \eqref{cKY}. Thus from Theorem \ref{fi-thm} we
see that normal solutions of equation \eqref{cKY} yield conformal
circle first integrals via \eqref{Sig-ex}.

In fact the requirement that the solution is normal can be dropped.
\begin{thm}\label{basic-ex-thm}
  On a pseudo-Riemannian manifold or conformal manifold, suppose that
  \( k_{ab}\in \Gamma(\ce_{[bc]}[3]) \) is a conformal Killing-Yano
  2-form, i.e.\ $k_{ab}$ satisfies
  \begin{align}\label{eq-CKY2}
    \nabla_{a}k_{bc} & =\nabla_{[a}k_{bc]} -\frac{2}{n-1}\mbf{g}_{a[b}\nabla^pk_{c]p} \, .
  \end{align}
  Write \(\Gamma( \Lambda^3\cT^* )\ni \mathbb{K}_{ABC}
  :=L(k) \) where $L$ is  the BGG splitting operator $L:
    \ce_{[bc]}[3]\to \mathcal{E}_{[ABC]}$, then expression
  \eqref{Sig-ex}, equivalently,
  $$
    \mathbf{u}^a \mathbf{a}^b k_{ab} \mp \frac{1}{n-1} \mathbf{u}^a \nabla^p k_{pa},
  $$
  is a first integral of a conformal circle with weighted velocity $\mbf{u}^a$ and acceleration $\mbf{a}^b$ with $\mbf{u}^a \mbf{u}_a = \pm 1$.
\end{thm}

\begin{proof}
  We assume the curve is spacelike or timelike, so that its weighted velocity $\mbf{u}^a$ satisfies $\mbf{u}^a\mbf{u}_a= 1$ or $\mbf{u}^a\mbf{u}_a= -1$ respectively,

  First, we explicitly compute the derivative of the quantity \eqref{Sig-ex} along the curve. In the conformal case we choose a metric $g\in\bc$ to compute.
  Computing $L(k)$ (cf.\ \cite{G-Sil-ckforms}) gives,
  \begin{equation}\label{kform-eq}
    \mathbb{K}_{ABC} = Y_{[A} Z_B {}^b Z_{C]} {}^c k_{bc} + Z_{[A} {}^a Z_B{}^b Z_{C]}{}^c \nabla_{a}k_{bc} + \frac{2}{n-1} X_{[A} Y_B Z_{C]}{}^a \nabla^p k_{pa} + X_{[A} Z_B {}^b Z_{C]} {}^c \rho_{bc},
  \end{equation}
  where
  \( \rho_{ab} \) will not be important for our purposes.

  Thus, using \eqref{cSigma}, we obtain
  \begin{align*}
    \Sigma^{ABC} \mathbb{K}_{ABC}
      & = \pm 6 \, \mathbf{u}^{c} X^A Y^B Z^C{}_c \mathbb{K}_{ABC} + 6 \, \mathbf{u}^b \mathbf{a}^c X^A Z^B {}_b Z^C {}_c \mathbb{K}_{ABC} \\
      & = 2 \, \mathbf{u}^a \mathbf{a}^b k_{ab} \mp \frac{2}{n-1} \mathbf{u}^a \nabla^p k_{pa}.
  \end{align*}
  Differentiating this and using \eqref{eq-wt-acc} and  \eqref{eq-proj-inv-wt2}, and the skew-symmetric of $k_{ab}$ leads to
  \begin{equation}
    \label{eq-RHS=0}
    \begin{split}
      \mathbf{u}^c \nabla_c \left( \Sigma^{ABC} \mathbb{K}_{ABC} \right)
      = \pm 2 \, \mathbf{u}^a \mathbf{u}^c
      & P_c {}^b k_{ab} + 2\,  \mathbf{u}^a \mathbf{a}^b \mathbf{u}^c \nabla_c k_{ab}                                            \\
      & \mp \frac{2}{n-1} \mathbf{a}^a \nabla^p k_{pa} \mp \frac{2}{n-1} \mathbf{u}^a \mathbf{u}^c \nabla_c \nabla^p k_{pa} \, .
    \end{split}
  \end{equation}
  Using \eqref{eq-CKY2} together with the fact that $\mbf{u}^a \mbf{u}_a = \pm 1$ and $\mbf{u}^b \mbf{a}_b = 0$ (see \eqref{eq-rel-wt-vel-acc}) shows that the two middle terms cancel. We compute
  \begin{align*}
    \mathbf{u}^a \mathbf{u}^c \nabla_c \nabla^p k_{pa} & =\mathbf{u}^a \mathbf{u}^c \nabla^p \nabla_c k_{pa} - (n-2) \mathbf{u}^a \mathbf{u}^c \Rho_c{}^p k_{pa}                       \\
                                                       & = \frac{1}{n-1} \mathbf{u}^a \mathbf{u}^c  \nabla_c \nabla^p k_{pa}  - (n-2) \mathbf{u}^a \mathbf{u}^c \Rho_c{}^p k_{pa} \, ,
  \end{align*}
  where we have commuted the covariant derivatives in the first line, and used \eqref{eq-CKY2} in the second line. Hence,
  \begin{align*}
    \mathbf{u}^a \mathbf{u}^c \nabla_c \nabla^p k_{pa} & = (n-1) \mathbf{u}^a \mathbf{u}^c \Rho_c{}^p k_{pa}
  \end{align*}
  from which  we conclude that the first and last terms of \eqref{eq-RHS=0} cancel each other out.
  Hence,
  \begin{align*}
    \mathbf{u}^c \nabla_c \left( \Sigma^{ABC} \mathbb{K}_{ABC} \right) & = 0 \, ,
  \end{align*}
  as required.
\end{proof}

\begin{rem}
  The quantity \eqref{Sig-ex} is a generalisation of Tod's quantity (c.f. \cite{Tod}, equations (17) and (B2)). Whereas the quantities of \cite{Tod} were constructed for specific 3- and 4-dimensional (pseudo-)Riemannian manifolds, \eqref{Sig-ex} exists on an arbitrary conformal manifold. When considering these specific cases, our equation simply recovers his.
\end{rem}

\subsubsection{An example from a trace-free part of $\otimes^2\Sigma$} \label{S-ex}

Let \( S^{AB} := -\frac{1}{2} \Sigma^{A} {}_{CD} \Sigma^{BCD} \) and write
$\mathring{S}^{AB}$ for the part that is trace-free with respect to
the conformal tractor metric. Since the tractor metric is parallel
everywhere and $\Sigma$ is parallel along any conformal circle it
follows at once that $S^{AB}$ and $\mathring{S}^{AB}$ are also
parallel along any conformal circle.  Thus if $H_{AB}\in
  \Gamma(\ce_{(AB)_0})$
is a parallel tractor on $(M,\bc)$ then
\begin{equation}\label{Sfi}
  \mathring{S}^{AB}H_{AB} ={S}^{AB}H_{AB}
\end{equation}
is a first integral for any conformal circle $\gamma$.
This is an example illustrating Theorem \ref{fi-thm}.

Let us write this explicitly in terms of the weighted velocity,
acceleration and the normal BGG solution corresponding to $H_{AB}$.
Since the projecting part of $\ce_{(AB)_0}$ is the density bundle
$\ce[2]$ (recovered by the map $H_{AB}\mapsto X^AX^BH_{AB}$) parallel
sections of $\ce_{(AB)_0}$ are equivalent to normal solutions from a
first BGG operator on $\ce[2]$. The latter is a 3rd order operator
$\mathcal{D}_0:\ce[2]\to \ce_{(abc)_0}[2]$,
the (conformally invariant) equation of which is
explicitly given by
\begin{equation} \label{cto}
  \nabla_{(a} \nabla_b \nabla_{c)_{0}} \tau + 4 \, \Rho_{(ab} \nabla_{c)_{0}} \tau + 2 \, \tau \nabla_{(a} \Rho_{bc)_{0}} = 0,
\end{equation}
for any $g\in\bc$ with Levi-Civita $\nabla$. The corresponding BGG splitting operator is a fourth-order differential operator
$$
  L:\ce[2]\to \ce_{(AB)_0}
$$
that takes the form
\begin{equation}\label{Lform}
  \begin{split}
    L(\tau)_{AB} =&
    \, Y_A Y_B \tau + Y_{(A}^{}Z_{B)}{}^{b} \na_b \tau
    + \frac12 Z_{(A}{}^{a} Z_{B)}{}^{b} \bigl[ \na_a \na_b \tau + 2\Rho_{ab}\tau
      -\frac{1}{n+2} \bg_{ab} \bigl( \Delta \tau + 2\mathsf{J} \tau \bigr)   \bigr] \\
    & \qquad - \frac{1}{n+2} X_{(A} Y_{B)} \bigl[ \Delta \tau + 2\mathsf{J} \tau \bigr]
    - X_{(A}^{}Z_{B)}{}^{b} \bigl[ \frac{1}{n+2} \na_b \bigl( \Delta + 2\mathsf{J} \bigr)
      \tau + \Rho_a{}^r \na_r\tau \bigr] \\
    & \qquad \qquad+ X_A X_B \bigl[\star],
  \end{split}
\end{equation}
where $\mathsf{J}:=\mathbf{g}^{ab} \Rho_{ab}$ (and we do not need the form of the $XX$ term.) The right hand side of \eqref{Lform} is parallel if and only if $\tau$ is a normal solution of \eqref{cto} (see Section \ref{BGGsec}). In particular $H=L(\tau)$ for some $\tau\in \Gamma (\ce[2])$.
On the other hand using \eqref{cSigma} \( S^{AB} \) is found to be
\begin{equation}\label{Sexpl}
  S^{AB} = \mathbf{u}^a \mathbf{u}^b Z^A {}_a Z^B {}_b \pm 2 X^{(A}Y^{B)} -2\mathbf{a}^b X^{(A} Z^{B)} {}_b \mp (\mathbf{a}^c \mathbf{a}_c) X^A X^B,
\end{equation}
where $\mathbf{u}^a$ and $\mathbf{a}^c$ are the weighted velocity and acceleration of the unparametrised curve $\gamma$ respectively, and $\mathbf{u}^a \mathbf{u}_a = \pm1$.
Thus, the conformal circle  first integral \eqref{Sfi} is explicitly given by
\begin{equation}\label{Sfi-expl}
  \begin{split}
    S^{AB} H_{AB} = \frac{1}{2} \mathbf{u}^a \mathbf{u}^b( & \nabla_a \nabla_b \tau + 2\Rho_{ab} \tau - \frac{1}{n+2} \mathbf{g}_{ab}(\Delta\tau + 2 \mathsf{J} \tau)) \\ &\mp \frac{2}{n+2}(\Delta \tau + 2 \mathsf{J} \tau ) - 2 \mathbf{a}^b \nabla_b \tau \mp (\mathbf{a}^b \mathbf{a}_b) \tau,
  \end{split}
\end{equation}
in terms of a metric $g\in \bc$.

As for the earlier examples a stronger result is available. The
normality is not needed:
\begin{thm}\label{S-thm}
  If $\tau\in \Gamma(\ce[2])$ is any solution of \eqref{cto} then
  \eqref{Sfi-expl} is a first integral for unparametrised conformal circles.
\end{thm}
\begin{proof}
  Suppose that $\tau\in \Gamma(\ce[2])$ is a solution of \eqref{cto}
  and that $\gamma$ is an unparametrised conformal circle with
  weighted velocity $\mbf{u}$ and weighted acceleration
  $\mbf{a}$. Then $S^{AB}=  -\frac{1}{2} \Sigma^{A} {}_{CD} \Sigma^{BCD}$
  is parallel along $\gamma$ and given by \eqref{Sexpl}.  So setting $H:=L(\tau)$ we have
  \[
    \mathbf{u}^c \nabla_c \left( S^{AB} H_{AB} \right) = \mathbf{u}^c S^{AB} \nabla_c H_{AB}.
  \]
  From the condition $\partial^*\nabla L(\tau)=0$ that in part defines $L$, or alternatively by direct calculation, it follows  that  \( \nabla H \) takes the form

  \[
    \nabla_c H_{AB} = \kappa_{cab} Z_A {}^a Z_B {}^b + \alpha_{bc}
    X_{(A}Z_{B)} {}^b + \omega_c X_A X_B,
  \]
  for some weighted tensors $\alpha_{bc}$, $\omega_c$ and $\kappa_{cab}=\kappa_{cba}$.
  Note that upon contraction with \( S^{AB} \), all terms of this display are annihilated except for the \( Z_A ^a Z_B ^b \) term.
  Thus
  \[
    \mathbf{u}^c S^{AB} \nabla_c H_{AB} = \mathbf{u}^a \mathbf{u}^b \mathbf{u}^c \kappa_{abc}.
  \]

  Now, using  \eqref{Lform} and the tractor connection formulae one calculates that
  \begin{equation} \label{ZZ-term-of-nabla-H}
    \begin{split}
      \kappa_{cab} = \frac{1}{2} [ \nabla_c & \nabla_a \nabla_b
        \tau + 2 (\nabla_c \Rho_{ab})\tau + 2 \Rho_{ab} \nabla_c \tau
        + 2  \Rho_{c(a}\nabla_{b)} \tau\\ &
        -\frac{1}{n+2} \mathbf{g}_{ab} \nabla_c (\Delta \tau + 2
        \mathsf{J} \tau ) - \frac{2}{n+2} \mathbf{g}_{c(a} \nabla_{b)}
        (\Delta \tau + 2 \mathsf{J} \tau ) - 2 \mathbf{g}_{c(a} \Rho_{b)}{}^d
        \nabla_d \tau]
    \end{split}
  \end{equation}
  Computing reveals that \eqref{ZZ-term-of-nabla-H} is trace-free over any pair of indices.
  Moreover
  contracting \( \mathbf{u}^a \mathbf{u}^b \mathbf{u}^c\) into this display will force symmetrisation over \( abc \). Thus
  $$
    \mathbf{u}^c S^{AB} \nabla_c H_{AB} = \mathbf{u}^a \mathbf{u}^b \mathbf{u}^c \kappa_{(abc)_0}.
  $$
  But $\kappa_{(abc)_0}$ is exactly $\cD_0(\tau)$, as given by the left hand side of \eqref{cto}. Thus $\mathbf{u}^c\nabla_c (S^{AB} H_{AB}) =0$ as claimed.
\end{proof}

\subsubsection{An example from of the general procedure} \label{genex}

The examples of expression (\ref{Sig-ex}) and expression (\ref{Sfi})
each illustrate cases that arise from the general procedure described
in Theorem \ref{genid} of section \ref{gen_proc}. (Moreover
\ref{key-if} is an analogue for geodesics.)

To see how Theorem \ref{genid} yields non-trivial first integrals one
need not necessarily push through the examples in full
detail. (Although these details can be computed completely
algorithmically, the computations can become demanding without the use
of software.) We illustrate this with  an
example from another trace part (cf.\ (\ref{Sfi})) of $\Gamma(S^2 \Lambda^3
\cT^*)|_{\gamma}$.  Let $\mathbb{S} := \Sigma^{ABE} \Sigma^{CD}
   {}_{E}$.  From \cite{Bailey1990a}, we know that there is a scale in
   the conformal class for which the conformal circle $\gamma$ is an
   affinely-parametrised geodesic and $u^a P_{ab} = 0$, where $u^a$ is
   the (unweighted) velocity of the curve $\gamma$, and $P_{ab}$ is
   the Schouten tensor for this special scale.  From
   (\ref{eq-wt-uwt-acc}), if we work in this scale we also have that
   $\mbf{a}^b = 0$, and hence the tractor $\Sigma$ takes the form
\begin{equation} \label{sigma-in-beastwood-scale}
  \Sigma = \pm6 \mbf{u}^c X^{[A} Y^B Z^{C]} {}_c,
\end{equation}
and therefore 
\begin{equation} \label{S_ABCD}
  \mathbb{S}^{ABCD} = \mbf{u}^e \mbf{u}^f \left( 4X^{[A} Y^{B]} X^{[C} Y^{D]} \mbf{g}_{ef} - 4 Y^{[A} Z^{B]} {}_e X^{[C} Z^{D]} {}_f - 4X^{[A} Z^{B]} {}_e Y^{[C} Z^{D]} {}_f \right),
\end{equation}
which is a section of $\cT^{[2,2]}$, in the notation of section
\ref{young-sec}.  An obvious way to make an irreducible part of this
section is to project to the Cartan part $\mathring{\cT}^{[2,2]}$ of
$\cT^{[2,2]}$.  This amounts to removing all traces to ensure that the
resulting section is totally trace-free.  Write
$\mathring{\mathbb{S}}^{ABCD}\in \Gamma(\mathring{\cT}^{[2,2]})$ for
the Cartan part of $\mathbb{S}^{ABCD}$.  To calculate
$\mathring{\mathbb{S}}^{ABCD}$ explicitly, we need the following
trace-part of $\mathbb{S}^{ABCD}$:

\begin{equation} \label{S^AB}
  S^{AB} := \mathbb{S}^{AEB} {}_E = -\mbf{u}^e \mbf{u}^f \left( 4X^{(A} Y^{B)} \mbf{g}_{ef} + 2Z^A {}_e Z^B {}_f \right)
  = \mp 4 X^{(A} Y^{B)} -2 Z^A {}_e Z^B {}_f \mbf{u}^e \mbf{u}^f 
\end{equation}
In terms of this trace part, one then has
\begin{align*}
  &\mathring{\mathbb{S}}^{ABCD} = \mathbb{S}^{ABCD} - \frac{1}{n} \left( S^{AC} h^{BD} - S^{BC} h^{AD} + S^{BD} h^{AC} - S^{AD} h^{BC} \right) \\
                               &+ \frac{1}{n(n+1)} S^{EF} {}_{EF} \left( h^{AC} h^{BD} - h^{BC} h^{AD} \right)\\
                               &= \pm4 X^{[A}Y^{B]} X^{[C}Y^{D]} - 4\mbf{u}^e \mbf{u}^f Y^{[A}Z^{B]} {}_e X^{[C}Z^{D]} {}_f - 4 \mbf{u}^e \mbf{u}^f X^{[A}Z^{B]} {}_e Y^{[C}Z^{D]}_f\\
                               &-\frac{1}{n} \left[ \left( \mp8 X^{(A}Y^{C)} X^{(B}Y^{D)}  \mp 4X^{(A}Y^{C)} Z^B {}_a Z^D {}_b \mbf{g}^{ab} - 4 X^{(B}Y^{D)} Z^A {}_e Z^C {}_f \mbf{u}^e \mbf{u}^f - 2 Z^A {}_e Z^C {}_f Z^B {}_a Z^D {}_b \mbf{g}^{ab} \mbf{u}^e \mbf{u}^f \right) \right.\\
                               &-\left( \mp8 X^{(B}Y^{C)} X^{(A}Y^{D)}  \mp 4X^{(B}Y^{C)} Z^A {}_a Z^D {}_b \mbf{g}^{ab} - 4 X^{(B}Y^{C)} Z^A {}_e Z^D {}_f \mbf{u}^e \mbf{u}^f - 2 Z^B {}_e Z^C {}_f Z^A {}_a Z^D {}_b \mbf{g}^{ab} \mbf{u}^e \mbf{u}^f \right)\\ 
                               &+\left( \mp8 X^{(B}Y^{D)} X^{(A}Y^{C)}  \mp 4X^{(B}Y^{D)} Z^A {}_a Z^C {}_b \mbf{g}^{ab} - 4 X^{(B}Y^{D)} Z^A {}_e Z^C {}_f \mbf{u}^e \mbf{u}^f - 2 Z^B {}_e Z^D {}_f Z^A {}_a Z^C {}_b \mbf{g}^{ab} \mbf{u}^e \mbf{u}^f \right)\\
                               &\left. -\left( \mp8 X^{(A}Y^{D)} X^{(B}Y^{C)}  \mp 4X^{(A}Y^{D)} Z^B {}_a Z^C {}_b \mbf{g}^{ab} - 4 X^{(A}Y^{D)} Z^B {}_e Z^C {}_f \mbf{u}^e \mbf{u}^f - 2 Z^A {}_e Z^D {}_f Z^B {}_a Z^C {}_b \mbf{g}^{ab} \mbf{u}^e \mbf{u}^f \right)\right]\\
                               &-\frac{4}{n(n+1)} \left[ \left( 4X^{(A}Y^{C)} X^{(B}Y^{D)} + 2 X^{(A}Y^{C)} Z^B {}_a Z^D {}_b \mbf{g}^{ab} \right. \right. \\
                               &\left. \left.+ 2 X^{(B}Y^{D)} Z^A {}_a Z^C {}_b \mbf{g}^{ab} + Z^A {}_a Z^C {}_c Z^B {}_b Z^D {}_d \mbf{g}^{ac} \mbf{g}^{bd} \right)\right.\\
                               & -  \left( 4X^{(B}Y^{C)} X^{(A}Y^{D)} + 2 X^{(B}Y^{C)} Z^A {}_a Z^D {}_b \mbf{g}^{ab} \right. \\
                               &\left. \left.+ 2 X^{(B}Y^{D)} Z^A {}_a Z^C {}_b \mbf{g}^{ab} + Z^B {}_a Z^C {}_c Z^A {}_b Z^D {}_d \mbf{g}^{ac} \mbf{g}^{bd} \right) \right]
\end{align*}

 This is clearly nonzero, as can be seen by
contracting with e.g. $X_A Z_B {}^p Y_C Z_D {}^q$.  Thus it follows
immediately that any non-zero parallel section of
$(\mathring{\cT}^{[2,2]})^*\cong \mathring{\cT}^{[2,2]}$ will pair
with $\mathring{\mathbb{S}}$ to yield a first integral of conformal
circles that is generically non-trivial.

\section{Distinguished curves as zero loci}\label{zero-sec}

The curve characterisations of Theorems \ref{main-p}, \ref{main-nullc} and \ref{main-c} lead to the conclusion that for suitable BGG solutions
the zero locus of part of the solution jet describes a distinguished curve.
This uses the curved orbit Theorem 2.6 of \cite{CGH-Duke}. What that
result shows is that on a parabolic geometry a parallel tractor field
determines a stratification of the underlying manifold, where the
different strata are in general initial submanifolds, with different
Cartan geometries induced on the strata components. Moreover, and what is most important for us here, by a
comparison map it is shown that locally there is a diffeomorphism
between the given underlying manifold and the model which maps the
strata to the corresponding strata on the model. This means that, for
example, if on the model a given stratum is an embedded smooth
submanifold then any corresponding stratum in the curved parabolic
geometry must necessarily also be an embedded smooth submanifold of
the same dimension. In the case of conformal and projective geometry
the stratification is determined entirely by the algebraic relation of
the canonical tractor $X$ to the given parallel
tractor.

\subsection{Conformal equations with distinguished curves as zero loci}

\begin{prop}\label{cc-zero}
  On a connected conformal manifold $(M,\bc)$ let $k_{bc}$ be a normal solution
  of the conformal Killing form equation such that the parallel tractor
  $L(k)\in \Gamma(\Lambda^3\cT)$ is simple, and of signature $(+,+,-)$
  or $(-,-,+)$.
  Then the locus of  points where
  $$
    (k_{bc},\nabla_{[a}k_{bc]}) \qquad \mbox{for any $g\in \bc$ with Levi-Civita $\nabla$,}
  $$
  both vanish is either empty or a conformal circle.
\end{prop}
\begin{proof}
  Since $k$ is a normal solution, the image $L(k)$ of the BGG
  splitting operator is parallel for the normal conformal tractor
  connection.  Note that $L(k)$ is a section of $\Lambda^3\cT$, see \eqref{kform-eq}.
  From the formula for $L(k)$ in a scale the condition
  $(k_{bc},\nabla_{[a}k_{bc]}) =0$ at some point $\mr{x}\in M$ is the same as
  $$
    X_{\mr{x}}\wedge L(k)_{\mr{x}}=0
  $$
  where $X_{\mr{x}}$ is the canonical tractor at $\mr{x}$.

  In the case of the model, if $\mathbb{K}$ is a parallel simple $3$-tractor of signature
  $(+,+,-)$
  or $(-,-,+)$ and
  $X \wedge \mathbb{K}$ is zero at some point $\mr{x}$, then  $X\wedge \mathbb{K}$ is zero
  along a curve through $\mr{x}$ (namely the unique conformal circle through
  $x$ with characterising 3-tractor $\Sigma:=\mathbb{K}$ at $\mr{x}$,
  see Section \ref{circ-model}.)  From \cite[Theorem 2.6]{CGH-Duke} it follows  that on
  $(M,\bc)$ for a $3$-tractor $\mathbb{K}$ that is parallel and of the same algebraic type
  (i.e. simple and of signature
  $(+,+,-)$
  or $(-,-,+)$) the zero locus of
  $X\wedge \mathbb{K}$ is either empty or is locally, and hence
  globally, an embedded curve. If the latter then it must be a conformal circle by Theorem \ref{main-c}.

  Thus, in particular, the zero locus of $X\wedge L(k)$ is either empty or  a conformal circle.
\end{proof}

By essentially the same argument we get the corresponding result for
null geodesics in indefinite conformal manifolds  as follows.

\begin{prop}\label{nullzero}
  On a connected indefinite conformal manifold $(M,\bc)$ let $k_{b}$ be
  a normal solution of the conformal Killing equation
  \eqref{cKt}, i.e.\ $\nabla_{(a}k_{b)_0}=0$, such that the parallel tractor
  $L(k)\in \Gamma(\Lambda^2\cT)$ is simple and totally null as in
  \eqref{nil-eq}. Then
  the locus of  points where
  $$
    (k_{b},\nabla_{[a}k_{b]}) \qquad \mbox{for any $g \in \bc$ with Levi-Civita $\nabla$,}
  $$
  both vanish is either empty or a null geodesic.
\end{prop}
\begin{proof}
  For a conformal Killing vector field $k$ the image $L(k)$ of the BGG
  splitting operator is a section of $\Lambda^2\cT$ and is given explicitly in
  e.g. \cite{GLaplacianEinstein,G-Sil-ckforms,HSSS}. From any of these  it is seen that
  the vanishing of $ X \wedge L(k) =0$ at some point $\mr{x}$
  is equivalent to $ (k_{b},\nabla_{[a}k_{b]})(\mr{x})=0$, where
  $\nabla$ is the Levi-Civita for any $g\in \bc$.  The argument
  otherwise proceeds  as the proof of Proposition
  \ref{cc-zero} above,  mutatis mutandis.
\end{proof}

\subsection{Projective geodesics and weighted bivectors}\label{proj-zero-sec}

Now we work on a projective manifold $(M,\bp)$.
In this case the relevant first BGG equation is
\begin{align}\label{eq-proj-wt-bv}
  \nabla_a \sigma^{bc} - 2 \delta^{[b}_a \tau^{c]} & = 0 \, ,
\end{align}
where  $\sigma^{ab}\in \Gamma(\mc{E}^{[ab]}(-2))$
and hence  $\tau^a := \frac{1}{n-1} \nabla_b \sigma^{ba}$. The prolonged system for this equation is given by \eqref{eq-proj-wt-bv} together with the equation
\begin{align*}
  \nabla_a \tau^b + \Rho_{ac} \sigma^{bc} + \frac{1}{2(n-2)} \sigma^{cd} W_{cd}{}^b{}_a & = 0 \, ,
\end{align*}
which must hold for any solution. The BGG splitting operator can then be computed to be
\begin{equation}\label{L-si}
  \Gamma(\mc{E}^{[ab]}(-2))\ni \si^{ab}\mapsto L(\si)= W^\alpha{}_a W^\beta{}_b \sigma^{ab} + \frac{2}{n-1}X^{[\alpha} W^{\beta]}{}_a \nabla_b \sigma^{ba} \in \Gamma(\Lambda^2\cT),
\end{equation}
where $\cT$ is the standard projective tractor bundle.

Thus we have the following:
\begin{prop} \label{f-zero}
  Let $\sigma^{ab} \in \Gamma( \mc{E}^{[ab]}(-2) )$ be a normal solution of
  \eqref{eq-proj-wt-bv}  such that the corresponding parallel
  $2$-tractor $\Sigma^{\alpha\beta}=L(\si)$ is simple. Then the zero set
  of $\sigma^{ab}$ is either empty or is an unparametrised  geodesic.
\end{prop}
\begin{proof}
  From \eqref{L-si} the vanishing of $X\wedge L(\si)$ at some point $\mr{x}$ is the same as $\si(\mr{x})=0$.  Otherwise the argument
  again proceeds via an obvious adaption of the proof of Proposition
  \ref{cc-zero} above.
\end{proof}

\bibliography{GST}

\Addresses

\end{document}